%% file: basechange.tex
\pdfoutput=1
\documentclass[leqno]{article}
\usepackage[normal]{phaine_style}
\usepackage[margin=1.5in]{geometry}
\input{document_macros}

\title{\Large From nonabelian basechange to basechange with coefficients}

\author{\normalsize Peter J. Haine}

\date{\normalsize \today}

\begin{document}

\maketitle

\begin{abstract} 
	The goal of this paper is to explain when basechange theorems for sheaves of spaces imply basechange for sheaves with coefficients in other presentable \categories.
	We accomplish this by analyzing when the tensor product of presentable \categories preserves left adjointable squares.
	As a sample result, we show that the Proper Basechange Theorem in topology holds with coefficients in any presentable \category which is compactly generated or stable.
	We also prove results about the interaction between tensor products of presentable \categories and various categorical constructions that are of independent interest.
\end{abstract}

\setcounter{tocdepth}{2}

\tableofcontents


\setcounter{section}{-1}

\section{Introduction}

Let
\begin{equation}\label{sq:LCHpull}
	\begin{tikzcd}
		W \arrow[r, "\fbar"] \arrow[d, "\gbar"'] \arrow[dr, phantom, very near start, "\lrcorner", xshift=-0.25em, yshift=0.25em] & Y \arrow[d, "g"] \\ 
		X \arrow[r, "f"'] & Z 
	\end{tikzcd}
\end{equation}
be a pullback square of locally compact Hausdorff topological spaces, and assume that the map $ g $ is proper.
The classical Proper Basechange Theorem in topology \cites[\stackstag{09V6}]{stacksproject}[Exposé Vbis, Théorème 4.1.1]{MR50:7131} says that for any ring $ R $, the induced square of bounded-above%
\footnote{We use \textit{homological} indexing.
What we write as $ \D(T;R)_{<\infty} $ is often written as $ \D^+(T;R) $.}
derived \categories 
\begin{equation}\label{sq:DLCHpull}
	\begin{tikzcd}
		\D(W;R)_{<\infty} \arrow[r, "\Rfbarlowerstar"] \arrow[d, "\Rgbarlowerstar"'] & \D(Y;R)_{<\infty} \arrow[d, "\Rglowerstar"] \\ 
		\D(X;R)_{<\infty} \arrow[r, "\Rflowerstar"'] & \D(Z;R)_{<\infty} 
	\end{tikzcd}
\end{equation}
is \textit{left adjointable}.
That is to say, for each object $ F \in \D(Y;R)_{<\infty} $, the natural \textit{exchange morphism}
\begin{equation*}
	\fromto{\Lfupperstar \Rglowerstar(F)}{\Rgbarlowerstar \Lfbarupperstar(F)}
\end{equation*}
is an equivalence.
As Lurie remarks \HTT{Remark}{7.3.1.19}, the classical Proper Basechange Theorem follows from the \textit{Nonabelian Proper Basechange Theorem} \HTT{Corollary}{7.3.1.18}: the induced square of \categories of sheaves of \textit{spaces} 
\begin{equation*}\label{sq:ShLCHpull}
	\begin{tikzcd}
		\Sh(W;\Spc) \arrow[r, "\fbarlowerstar"] \arrow[d, "\gbarlowerstar"'] & \Sh(Y;\Spc) \arrow[d, "\glowerstar"] \\ 
		\Sh(X;\Spc) \arrow[r, "\flowerstar"'] & \Sh(Z;\Spc) 
	\end{tikzcd}
\end{equation*}
is left adjointable.

The goal of this paper is to expand on Lurie's remark and explain when basechange results for sheaves of spaces imply basechange results for sheaves with coefficients in other presentable \categories.
Our inquiry is informed by the following observation: for a topological space $ T $ and ring $ R $, the \textit{unbounded} derived \category $ \D(T;R) $ naturally embeds as a full subcategory of the Deligne--Lurie tensor product of presentable \categories
\begin{equation*}
	\Sh(T;\D(R)) \colonequals \Sh(T;\Spc) \tensor \D(R) 
\end{equation*}
\cite[\SAGthm{Remark}{1.3.1.6}, \SAGthm{Corollary}{1.3.1.8}, \& \SAGthm{Corollary}{2.1.2.3}]{SAG}.
That is, $ \D(T;R) $ embeds into the \category of sheaves on $ T $ valued in the derived \category of $ R $.
Moreover:
\begin{enumerate}
	\item The essential image of this embedding $ \incto{\D(T;R)}{\Sh(T;\D(R))} $ is the full subcategory spanned by the $ \D(R) $-valued \textit{hyper}sheaves on $ T $.
	In many situations the two \categories coincide, e.g., if $ T $ admits a CW structure \cite{MO:168526} or is sufficiently finite-dimensional \cites[\HTTthm{Corollary}{7.2.1.12}, \HTTthm{Theorem}{7.2.3.6} \& \HTTthm{Remark}{7.2.4.18}]{HTT}[Theorem 3.12]{ClausenMathew:hyperdescent}.

	\item There is a natural \tstructure on the stable \category $ \Sh(T;\D(R)) $. 
	Moreover, the embedding $ \incto{\D(T;R)}{\Sh(T;\D(R))} $ is \texact and restricts to an equivalence
	\begin{equation*}
		\equivto{\D(T;R)_{<\infty}}{\Sh(T;\D(R))_{<\infty}}
	\end{equation*}
	on \tboundedabove objects \SAG{Corollary}{2.1.2.4}.
\end{enumerate}
These points raise a natural question:

\begin{question}\label{quest:motivating}
	Does the Proper Basechange Theorem hold with the bounded-above dervied \categories $ \D(-;R)_{<\infty} $ replaced by the larger \categories $ \Sh(-;\D(R)) $?
	If so, can this extension of the be deduced from Lurie's Nonabelian Proper Basechange Theorem by a `formal' argument about the tensor product of presentable \categories preserving adjointability?
\end{question}

We explain why the answer to both questions is affirmative.
However, there are some important subtleties.
The general setup we consider is a square of presentable \categories and right adjoints
\begin{equation}\label{eq:introgensq}
	\begin{tikzcd}
		A \arrow[r, "\fbarlowerstar"] \arrow[d, "\gbarlowerstar"'] & C \arrow[d, "\glowerstar"] \arrow[dl, phantom, "\scriptstyle \sigma" below right, "\Longleftarrow" sloped] \\ 
		B \arrow[r, "\flowerstar"'] & D 
	\end{tikzcd}
\end{equation}
equipped with a (not necessarily invertible) natural transformation $ \sigma \colon \fromto{\glowerstar \fbarlowerstar}{\gbarlowerstar\flowerstar} $.
In this general setting, there is a natural \textit{exchange morphism}
\begin{equation*}
	\Ex_{\sigma} \colon \fromto{\fupperstar \glowerstar}{\gbarlowerstar \fbarupperstar}
\end{equation*}
associated to the diagram \eqref{eq:introgensq}; see \cref{subsec:adjointability}.
Let $ E $ be another presentable \category.
The main subtlety is that even if the exchange morphism $ \fromto{\fupperstar \glowerstar}{\gbarlowerstar \fbarupperstar} $ is an equivalence, the exchange morphism associated to the tensored-up diagram
\begin{equation}\tag*{$ (\textup{\ref{eq:introgensq}}) \tensor E $}\label{eq:introgensqE}
	\begin{tikzcd}[sep=3em]
		A \tensor E \arrow[r, "\fbarlowerstar \tensor E"] \arrow[d, "\gbarlowerstar \tensor E"'] & C \tensor E \arrow[d, "\glowerstar \tensor E"] \arrow[dl, phantom, "\scriptstyle \sigma \tensor E" below right, "\Longleftarrow" sloped] \\ 
		B \tensor E \arrow[r, "\flowerstar \tensor E"'] & D \tensor E
	\end{tikzcd}
\end{equation}
need not be an equivalence (see \Cref{ex:trunSpcsquare}).

However, there are many situations in which the left adjointability of \eqref{eq:introgensq} implies the left adjointability of \ref*{eq:introgensqE}. 
The following is probably the most useful result in this direction; when $ E $ is compactly generated, this applies to squares of \topoi and geometric morphisms.

\begin{theorem}[(\Cref{cor:CGadjointability,cor:stabtensorfilteredcolim})]\label{thm:main}
	Consider an oriented square \eqref{eq:introgensq} of presentable \categories and right adjoints.
	Assume that the left adjoints $ \fupperstar $ and $ \fbarupperstar $ are left exact and that square \eqref{eq:introgensq} is left adjointable.
	Let $ E $ be a presentable \category, and assume that one of the following conditions is satisfied:
	\begin{enumerate}[label=\stlabel{thm:main}, ref=\arabic*]
		\item The \category $ E $ is compactly generated.

		\item The \category $ E $ is stable and the right adjoints $ \glowerstar $ and $ \gbarlowerstar $ preserve filtered colimits.
	\end{enumerate}
	Then the induced square \ref*{eq:introgensqE} is left adjointable.
\end{theorem}

Throughout this paper, we also prove other adjointability results as well as results about the interaction between tensor products of presentable \categories and various categorical constructions that are of independent interest.

\begin{example}[(\Cref{subex:propertopology})]
	Let us return to the setting of a pullback square of locally compact Hausdorff spaces \eqref{sq:LCHpull} where the morphism $ g \colon \fromto{Y}{Z} $ is proper.
	Lurie's Nonabelian Proper Basechange Theorem and \Cref{thm:main} show that if $ E $ is a presentable \category which is compactly generated or stable, then the induced square of \categories of $ E $-valued sheaves
	\begin{equation}\label{sq:ShLCHpullE}
		\begin{tikzcd}
			\Sh(W;E) \arrow[r, "\fbarlowerstar"] \arrow[d, "\gbarlowerstar"'] & \Sh(Y;E) \arrow[d, "\glowerstar"] \\ 
			\Sh(X;E) \arrow[r, "\flowerstar"'] & \Sh(Z;E) 
		\end{tikzcd}
	\end{equation}
	is left adjointable.
	This generalizes the classical Proper Basechange Theorem in two important ways:
	\begin{enumerate}
		\item Let $ R $ be an ordinary ring, and let $ E = \D(R) $ be the unbounded derived \category of $ R $. 
		The left adjointability of the square \eqref{sq:ShLCHpullE} generalizes the classical Proper Basechange Theorem to objects of $ \Sh(T;\D(R)) $ that are not bounded-above, and answers \Cref{quest:motivating} in the affirmative.

		\item A version of the Proper Basechange Theorem holds for sheaves of modules over any $ \Eup_1 $-ring \textit{spectrum} $ R $ or \textit{animated ring} (in the terminology of \cites[Appendix A]{arXiv:2201.06120}[\S5.1.4]{arXiv:1912.10932}).
	\end{enumerate}
\end{example}

\begin{remark}[(unbounded derived \categories)]
	There are two natural squares of right adjoints enlarging the square \eqref{sq:DLCHpull} appearing in the classical Proper Basechange Theorem: the square of classical unbounded derived \categories
	\begin{equation}\label{sq:classicalD}
		\begin{tikzcd}
			\D(W;R) \arrow[r, "\Rfbarlowerstar"] \arrow[d, "\Rgbarlowerstar"'] & \D(Y;R) \arrow[d, "\Rglowerstar"] \\ 
			\D(X;R) \arrow[r, "\Rflowerstar"'] & \D(Z;R) 
		\end{tikzcd}
	\end{equation}
	and the square of \categories of $ \D(R) $-valued sheaves
	\begin{equation}\label{sq:ShLCHpullDR}
		\begin{tikzcd}
			\Sh(W;\D(R)) \arrow[r, "\fbarlowerstar"] \arrow[d, "\gbarlowerstar"'] & \Sh(Y;\D(R)) \arrow[d, "\glowerstar"] \\ 
			\Sh(X;\D(R)) \arrow[r, "\flowerstar"'] & \Sh(Z;\D(R)) \period
		\end{tikzcd}
	\end{equation}
	We have seen that the square \eqref{sq:ShLCHpullDR} is left adjointable.
	Moreover, for a topological space $ T $, the unbounded derived \category $ \D(T;R) $ is the full subcategory of $ \Sh(T;\D(R)) $ spanned by the hypersheaves.
	So the square \eqref{sq:ShLCHpullDR} is really an enlargement of the square \eqref{sq:classicalD}.
	However, the left adjointability of the square \eqref{sq:ShLCHpullDR} does \textit{not} imply the left adjointability of the square \eqref{sq:classicalD}, and the square \eqref{sq:classicalD} is \textit{not} generally left adjointable.
	See \cite[\HTTthm{Counterexample}{6.5.4.2} \& \HTTthm{Remark}{6.5.4.3}]{HTT}.
	The key point is the following: if $ F \in \D(Y;R) $ is not \tboundedabove, then the exchange transformation
	\begin{equation}\label{eq:LRexchange}
		\fromto{\Lfupperstar \Rglowerstar(F)}{\Rgbarlowerstar \Lfbarupperstar(F)}
	\end{equation}
	associated to the square \eqref{sq:classicalD} does \textit{not} agree with the exchange transformation
	\begin{equation*}
		\fromto{\fupperstar \glowerstar(F)}{\gbarlowerstar \fbarupperstar(F)}
	\end{equation*}
	associated to \eqref{sq:ShLCHpullDR}.
	The reason is that given a map of topological spaces $ p \colon \fromto{T}{S} $, the pullback functor
	\begin{equation*}
		\pupperstar \colon \fromto{\Sh(S;\D(R))}{\Sh(T;\D(R))}
	\end{equation*}
	does not generally carry $ \D(S;R) $ to $ \D(T;R) $.
	The inclusion $ \incto{\D(T;R)}{\Sh(T;\D(R))} $ admits a \texact left adjoint $ (-)^{\hyp} \colon \fromto{\Sh(T;\D(R))}{\D(T;R)} $ called \textit{hypercompletion}, and the left derived functor $ \Lpupperstar \colon \fromto{\D(S;R)}{\D(T;R)} $ is the composite
	\begin{equation*}
		\begin{tikzcd}
			\D(S;R) \arrow[r, hooked] & \Sh(S;\D(R)) \arrow[r, "\pupperstar"] & \Sh(T;\D(R)) \arrow[r, "(-)^{\hyp}"] & \D(T;R) \period
		\end{tikzcd}
	\end{equation*}
	This extra hypercompletion procedure is nontrivial and is what prevents the exchange transformation \eqref{eq:LRexchange} from being an equivalence in general.

	The \texact inclusion $ \incto{\D(T;R)}{\Sh(T;\D(R))} $ restricts to an equivalence on hearts; hence the \category $ \Sh(T;\D(R)) $ is not generally the derived \category of an abelian category.
	Thus, if one wants a version of the Proper Basechange Theorem for unbounded complexes, one is forced to leave the world of classical derived categories and needs to work with \categories.
	These comments are the reason Spaltenstein \cite{MR932640} was unable to prove a version of the Proper Basechange Theorem for arbitrary unbounded complexes.
	They also highlight a major advantage of working with the \categories $ \Sh(-;\D(R)) $ over the \categories $ \D(-;R) $.
\end{remark}

\begin{remark}
	We have been aware of \Cref{thm:main} for some time, and certainly results of this form are known to experts.
	However, we were unable to locate a source explaining the relationship between left adjointability and tensoring with a presentable \category.
	We have written this paper because we need to use results of this form in forthcoming work; we hope that others will also find the results presented here useful.
\end{remark}


\subsection*{Linear overview}\label{subsec:linoverview}

In \cref{sec:prelim,sec:CG}, we recall the background we need about adjointability and tensor products of presentable \categories.
We also collect key examples of when tensoring with a presentable \category does or does not preserve left adjointability.
Setting up the notation and explaining explicit descriptions of the tensor product we need takes a bit of time, but once everything is in place \Cref{thm:main} is elementary.
The key insight is that tensoring with a compactly generated \category has additional unexpected functoriality: it is functorial not only in adjunctions, but also arbitrary left exact functors.
We also use these explicit descriptions to show that the tensor product preserves many properties of functors (\cref{subsec:conservativity}), tensoring with a compactly generated \category preserves limits of diagrams of left exact left adjoints (\cref{subsec:commuting_tensors_past_limits}), and that, in most situations that arise in nature, the tensor product preserves recollements (\cref{subsec:recollements}).
In \cref{sec:results}, we prove \Cref{thm:main} and derive some consequences.
\Cref{sec:adjointability_stabilization} deals with situations where we only know that the exchange morphism is an equivalence when restricted to a (not necessarily presentable) subcategory.


\begin{acknowledgments}
	We thank Clark Barwick, Marc Hoyois, and Lucy Yang for insightful discussions.
	We thank Mauro Porta and Jean-Baptiste Teyssier for helpful correspondence and for encouraging us to include \cref{subsec:recollements}.
	Special thanks are due to Jacob Lurie for explaining \Cref{ex:Spcsquaregeneral,ex:trunSpcsquare}.

	We gratefully acknowledge support from the MIT Dean of Science Fellowship, the NSF Graduate Research Fellowship under Grant \#112237, UC President's Postdoctoral Fellowship, and NSF Mathematical Sciences Postdoctoral Research Fellowship under Grant \#DMS-2102957. 
\end{acknowledgments}


\subsection*{Terminology and notations}\label{subsec:terminology}

We use the terms \textit{\category} and \textit{$ (\infty,1) $-category} interchangeably.
We write $ \Spc $ for the \category of spaces.
Given \categories $ C $ and $ D $ with limits, we write \smash{$ \Funlim(C,D) \subset \Fun(C,D) $} for the full subcategory spanned by the limit-preserving functors $ \fromto{C}{D} $.

In this paper, we use a small amount of the theory of $ (\infty,2) $-categories; all of the $ (\infty,2) $-categories we use are subcategories of the $ (\infty,2) $-category $ \Cat_{\infty} $ of locally small but potentialy large $ (\infty,1) $-categories, functors, and natural transformations.
Moreover, all functors of $ (\infty,2) $-categories are subfunctors of the functor $ \goesto{(C,D)}{\Fun(C,D)} $.
We write $ \PrR \subset \Cat_{\infty} $ for the sub-$ (\infty,2) $-category of presentable $ (\infty,1) $-categories, \textit{right} adjoints, and all natural tranformations.
We write $ \PrL \subset \Cat_{\infty} $ for the sub-$ (\infty,2) $-category of presentable $ (\infty,1) $-categories, \textit{left} adjoints, and all natural tranformations.


\section{Preliminaries on adjointability \& tensor products}\label{sec:prelim}

In this section we recall the basics of left adjointable squares and tensor products with presentable \categories.
\Cref{subsec:adjointability} fixes our conventions on adjointability and gives some examples of adjointable squares.
\Cref{subsec:tensor} recalls tensor products of presentable \categories.
\Cref{subsec:tensoradj} gives an example explaining why tensoring does not generally preserve left adjointable squares of presentable \categories (\Cref{ex:trunSpcsquare}).
We also provide a class of left adjointable squares that are preserved by tensoring with any presentable \category (\Cref{prop:extrartadj}).


\subsection{Oriented squares \& adjointability}\label{subsec:adjointability}

We begin by fixing conventions for adjointability in an $ (\infty,2) $-category.

\begin{definition}
	Let $ \Cbf $ be an $ (\infty,2) $-category, and $ A $, $ B $, $ C $, and $ D $ objects of $ \Cbf $.
	We exhibit data of $ 1 $-morphisms $ \flowerstar \colon\fromto{B}{D}$, $ \glowerstar\colon\fromto{C}{D}$, $ \gbarlowerstar \colon\fromto{A}{B}$, and $ \fbarlowerstar \colon \fromto{A}{C}$, along with a $ 2 $-morphism $ \sigma \colon \fromto{\glowerstar \fbarlowerstar}{\gbarlowerstar\flowerstar} $ by a single square
	\begin{equation}\label{square:catBCsquare}
		\begin{tikzcd}
			A \arrow[r, "\fbarlowerstar"] \arrow[d, "\gbarlowerstar"'] & C \arrow[d, "\glowerstar"] \arrow[dl, phantom, "\scriptstyle \sigma" below right, "\Longleftarrow" sloped] \\ 
			B \arrow[r, "\flowerstar"'] & D \period
		\end{tikzcd}
	\end{equation}
	We refer to such a square as an \defn{oriented square} in $ \Cbf $.
\end{definition}

\begin{definition}\label{def:basechangeconditions}
	Let $ \Cbf $ be an $ (\infty,2) $-category and consider an oriented square \eqref{square:catBCsquare} in $ \Cbf $.
	\begin{enumerate}[label=\stlabel{def:basechangeconditions}, ref=\arabic*]
		\item\label{def:basechangeconditions.1} Assume that the $ 1 $-morphisms $ \flowerstar $ and $ \fbarlowerstar $ admit left adjoints $ \fupperstar $ and $ \fbarupperstar $, respectively.
		Write $ \counit_f \colon \fromto{\fupperstar \flowerstar}{\id{B}}$ for the counit and $ \unit_{\fbar} \colon \fromto{\id{C}}{\fbarlowerstar\fbarupperstar}$ for the unit.
		The \defn{left exchange transformation} associated to the oriented square \eqref{square:catBCsquare} is the composite $ 2 $-morphism
		\begin{equation*}
			\begin{tikzcd}[sep=2.75em]
				\Ex_{\sigma} \colon \fupperstar \glowerstar \arrow[r, "\fupperstar \glowerstar\unit_{\fbar}"] & \fupperstar \glowerstar\fbarlowerstar\fbarupperstar \arrow[r, "\fupperstar \sigma \fbarupperstar"] & \fupperstar \flowerstar\gbarlowerstar\fbarupperstar \arrow[r, "\counit_f \gbarlowerstar\fbarupperstar"] & \gbarlowerstar\fbarupperstar \period 
			\end{tikzcd}
		\end{equation*}	
		We say that the square \eqref{square:catBCsquare} is \defn{(horizontally) left adjointable} if the exchange transformation $ \Ex_{\sigma} \colon \fupperstar \glowerstar \to \gbarlowerstar\fbarupperstar $ is an equivalence.

		\item\label{def:basechangeconditions.2} Assume that the $ 1 $-morphisms $ \glowerstar $ and $ \gbarlowerstar $ admit right adjoints $ \guppersharp $ and $ \gbaruppersharp $, respectively.
		Write $ \counit_{\gbar} \colon \fromto{\gbarlowerstar \gbaruppersharp}{\id{B}} $ for the counit and $ \unit_g \colon \fromto{\id{C}}{\guppersharp \glowerstar}$ for the unit.
		The \defn{right exchange transformation} associated to the oriented square \eqref{square:catBCsquare} is the composite $ 2 $-morphism
		\begin{equation*}
			\begin{tikzcd}[sep=2.75em]
				\fbarlowerstar\gbaruppersharp \arrow[r, "\unit_g \fbarlowerstar\gbaruppersharp"] & \guppersharp\glowerstar\fbarlowerstar\gbaruppersharp \arrow[r, "\guppersharp\sigma \gbaruppersharp"] & \guppersharp\flowerstar\gbarlowerstar\gbaruppersharp \arrow[r, "\guppersharp\flowerstar \counit_{\gbar}"] & \guppersharp\flowerstar \period 
			\end{tikzcd}
		\end{equation*}
		We say that the square \eqref{square:catBCsquare} is \defn{(vertically) right adjointable} if the exchange transformation $ \fbarlowerstar\gbaruppersharp \to \guppersharp\flowerstar $ is an equivalence.
	\end{enumerate}
\end{definition}

\begin{remark}
	We follow Hoyois \cites{MR3302973}{MR3570135} in calling the morphism $ \Ex_{\sigma} $ the \textit{exchange transformation}.
	The natural transformation $ \Ex_{\sigma} $ is often referred to as a \textit{Beck--Chevalley transformation} \cites{MR255631}[\S2.2]{arXiv:1811.02057}[Notation 4.1.1]{HopkinsLurie:ambidexterity}, \textit{basechange transformation} \cite[Definition 7.1.1]{exodromy}, or \textit{mate transformation} \cites[\S1]{MR3189430}{arXiv:2011.08808}[\S2.2]{MR0357542}.
	Instead of $ \Ex $, the notations $ \BC $ (for Beck--Chevalley or basechange) and $ \beta $ are often used \cites[Definition 7.1.1]{exodromy}[\S2.2]{MR0357542}[Notation 4.1.1]{HopkinsLurie:ambidexterity}.
\end{remark}

\begin{remark}[(on notation)]
	The notation we have chosen is meant to provide an easy way to remember the specifics of left exchange transformations: the left exchange transformation goes from a composite with no bars to a composite with bars, and an oriented square is left adjointable if we can `exchange $ \fupperstar $ and $ \glowerstar $' at the cost of adding bars.
\end{remark}

In this paper, we are mostly concerned with \textit{left} adjointability, but right adjointability will also appear due to the following.

\begin{observation}\label{obs:adjointability}
	\hfill
	\begin{enumerate}[label=\stlabel{obs:adjointability}, ref=\arabic*]
		\item Since functors of $ (\infty,2) $-categories preserve adjunctions and their (co)units, functors of $ (\infty,2) $-categories preserve left/right adjointable oriented squares.

		\item If in the square \eqref{square:catBCsquare} $ \flowerstar $ and $ \fbarlowerstar $ admit left adjoints and $ \glowerstar $ and $ \gbarlowerstar $ admit right adjoints, then \eqref{square:catBCsquare} is horizontally left adjointable if and only if \eqref{square:catBCsquare} is vertically right adjointable.
	\end{enumerate}
\end{observation}

\begin{notation}\label{ntn:generalsquare}
	For the rest of this paper, we fix an oriented square of \categories
	\begin{equation}\tag{$ \mdlgwhtsquare $}\label{sq:generalsquare}
		\begin{tikzcd}
			A \arrow[r, "\fbarlowerstar"] \arrow[d, "\gbarlowerstar"'] & C \arrow[d, "\glowerstar"] \arrow[dl, phantom, "\scriptstyle \sigma" below right, "\Longleftarrow" sloped] \\ 
			B \arrow[r, "\flowerstar"'] & D \comma
		\end{tikzcd}
	\end{equation}
	where the functors $ \flowerstar $ and $ \fbarlowerstar $ admit left adjoints $ \fupperstar $ and $ \fbarupperstar $, respectively.
\end{notation}

\begin{convention}
	Unless explicitly stated otherwise, adjointability of an oriented square of (presentable) \categories \eqref{sq:generalsquare} refers to adjointability in the $ (\infty,2) $-category $ \Cat_{\infty} $.
\end{convention}

We finish this subsection with two examples of left adjointable squares. 
The first is an easy-to-state version of the Smooth and Proper Basechange Theorem in algebraic geometry.

\begin{example}
	Let $ k $ be an algebraically closed field and let
	\begin{equation*}
		\begin{tikzcd}
			W \arrow[r, "\fbar"] \arrow[d, "\gbar"'] \arrow[dr, phantom, very near start, "\lrcorner", xshift=-0.25em, yshift=0.25em] & Y \arrow[d, "g"] \\ 
			X \arrow[r, "f"'] & Z 
		\end{tikzcd}
	\end{equation*}
	be a pullback square of quasiprojective $ k $-schemes.
	Let $ \el $ be a prime number different from the characteristic of $ k $.
	The Smooth and Proper Basechange Theorem in étale cohomology says that if the morphism $ f $ is smooth or the morphism $ g $ is proper, then the induced square 
	\begin{equation*}
		\begin{tikzcd}
			\Sh_{\et}(W;\ZZ_\el) \arrow[r, "\fbarlowerstar"] \arrow[d, "\gbarlowerstar"'] & \Sh_{\et}(Y;\ZZ_\el) \arrow[d, "\glowerstar"] \\ 
			\Sh_{\et}(X;\ZZ_\el) \arrow[r, "\flowerstar"'] & \Sh_{\et}(Z;\ZZ_\el)
		\end{tikzcd}
	\end{equation*}
	of \categories of $ \el $-adic étale sheaves is left adjointable \cite[Theorem 2.4.2.1]{MR3887650}.
\end{example}

\begin{example}\label{ex:Spcsquaregeneral}
	Let $ X $ and $ Y $ be spaces, and consider the canonically commutative square of presentable \categories and right adjoints
	\begin{equation}\label{sq:Spcsquare}
		\begin{tikzcd}[column sep=3.5em, row sep=2em]
			\Spc \arrow[r, "{\Map(X,-)}"] \arrow[d, "{\Map(Y,-)}"'] & \Spc \arrow[d, "{\Map(Y,-)}"] \\ 
			\Spc \arrow[r, "{\Map(X,-)}"'] & \Spc \period
		\end{tikzcd}
	\end{equation}
	The associated exchange morphism
	\begin{equation}\label{eq:MapEx}
		X \cross \Map(Y,-) \longrightarrow \Map(Y,X \cross (-)) \equivalent \Map(Y,X) \cross \Map(Y,-)
	\end{equation}
	is the product of the diagonal map $ \fromto{X}{\Map(Y,X)} $ with the identity map on $ \Map(Y,-) $.
	In particular, the exchange morphism \eqref{eq:MapEx} is an equivalence if and only if the diagonal map $ \fromto{X}{\Map(Y,X)} $ is an equivalence.
\end{example}

\begin{nul}\label{rec:Sullivan}
	One version of the Sullivan Conjecture (proven by Carlsson \cite{MR1091616}, Lannes \cite{MR1179079}, and Miller \cites{MR699318}{MR750716}{MR794376}{MR934259}) says that if $ X $ is a finite space and $ Y $ is a connected \pifinite space, then the diagonal map $ \fromto{X}{\Map(Y,X)} $ is an equivalence.
	Thus, under these hypotheses, the square \eqref{sq:Spcsquare} is left adjointable.
\end{nul}


\subsection{Tensor products of presentable \texorpdfstring{$\infty$}{∞}-categories}\label{subsec:tensor}

\begin{recollection}\label{rec:HA.4.8.1.17}
	Let $ S $ and $ E $ be presentable \categories.
	The \defn{tensor product} of presentable \categories $ S \tensor E $ along with the functor
	\begin{equation*}
		\tensor \colon \fromto{S \cross E}{S \tensor E}
	\end{equation*}
	are characterized by the following universal property: for any presentable \category $ T $, restriction along $ \tensor $ defines an equivalence 
	\begin{equation*}
		\Fun^{\colim}(S \tensor E,T) \equivalence \Fun^{\colim,\colim}(S \cross E,T) 
	\end{equation*}
	between colimit-preserving functors $ \fromto{S \tensor E}{T} $ and functors $ \fromto{S \cross E}{T} $ that preserve colimits separately in each variable.
	The tensor product of presentable \categories defines a functor
	\begin{equation*}
		\tensor \colon \fromto{\PrL \cross \PrL}{\PrL}
	\end{equation*}
	and can be used to equip $ \PrL $ with the structure of a symmetric monoidal $ (\infty,2) $-category.
	See \cites[Chapter 1, \S6.1]{MR3701352}[\S4.4]{MR3607274} for this statement for presentable stable \categories; the proof is exactly the same without the stability hypothesis.

	Since the $ (\infty,2) $-category $ \PrR $ of presentable \categories and right adjoints is obtained from $ \PrL $ by reversing $ 1 $-morphisms and $ 2 $-morphisms, the tensor product also defines a symmetric monoidal structure on $ \PrR $.
	In this note, we are more interested in the tensor product on $ \PrR $.
	This has a very explicit description: there is a natural equivalence of \categories
	\begin{equation*}
		S \tensor E \equivalent \Funlim(E^{\op},S) 
	\end{equation*}
	\HA{Proposition}{4.8.1.17}.
	Moreover, there is a natural equivalence
	\begin{equation*}
		(-) \tensor E \equivalent \Funlim(E^{\op},-)
	\end{equation*}
	of functors of $ (\infty,2) $-categories $ \fromto{\PrR}{\PrR} $.
	In particular, if $ \plowerstar \colon \fromto{S}{T} $ is a right adjoint functor of presentable \categories, then the induced right adjoint
	\begin{equation*}
		\plowerstar \tensor E \colon S \tensor E \equivalent \Funlim(E^{\op},S) \to \Funlim(E^{\op},T) \equivalent T \tensor E
	\end{equation*}
	is given by post-composition with $ \plowerstar $.
	For the purposes of this work, it suffices to take
	\begin{equation*}
		\Funlim(E^{\op},-) \colon \fromto{\PrR}{\PrR}
	\end{equation*}
	as the \textit{definition} of the tensor product $ (-) \tensor E $.
\end{recollection}


\begin{observation}\label{obs:adjointcompat}
	Let $ h \colon \fromto{S}{S'} $ and $ v \colon \fromto{E}{E'} $ be functors between presentable \categories which are both left adjoints or both right adjoints.
	Then the square
	\begin{equation*}
		\begin{tikzcd}[sep=3em]
			S \tensor E \arrow[r, "h \tensor E"] \arrow[d, "S \tensor v"'] & S' \tensor E \arrow[d, "S' \tensor v"] \\ 
			S \tensor E' \arrow[r, "h \tensor E'"'] & S' \tensor E'
		\end{tikzcd}
	\end{equation*}
	canonically commutes: both composites are identified with $ h \tensor v $.
\end{observation}


\subsection{Interaction between tensor products and adjointability}\label{subsec:tensoradj}

Now we give an example showing that the functor $ (-) \tensor E \colon \fromto{\PrR}{\PrR} $ need not preserve left adjointability of oriented squares (so that \Cref{thm:main} is not completely trivial).
The problem here is that we are interested in adjointability in $ \Cat_{\infty} $, rather the much stronger notions of adjointability in $ \PrR $ or $ \PrL $.
Said differently, the composite functors $ \fupperstar \glowerstar $ and $ \gbarlowerstar\fbarupperstar $ involved in the exchange transformation are not generally right or left adjoints.
Hence the condition that the exchange morphism $ \fromto{\fupperstar \glowerstar}{\gbarlowerstar\fbarupperstar} $ be an equivalence is not expressible internally to $ \PrR $ or $ \PrL $, thus need not be preserved by the functor of $ (\infty,2) $-categories $ (-) \tensor E $.

The following example is a variant of \Cref{ex:Spcsquaregeneral}; we learned of it from Lurie.

\begin{example}\label{ex:trunSpcsquare}
	Let $ p $ be a prime number and write $ \BCp $ for the classifying space of the cyclic group $ \Cp $ of order $ p $.
	Let $ X $ be a connected finite space such that \smash{$ \uppi_1(X) \isomorphic \Cp $}.
	(For example, take $ p = 2 $ and \smash{$ X = \RP^2 $}.)
	As a special case of \cref{rec:Sullivan}, the square of presentable \categories
	\begin{equation}\label{sq:SpcsquareSullivan}
		\begin{tikzcd}[column sep=5.5em, row sep=3em]
			\Spc \arrow[r, "{\Map(X,-)}"] \arrow[d, "{\Map(\BCp,-)}"'] & \Spc \arrow[d, "{\Map(\BCp,-)}"] \\ 
			\Spc \arrow[r, "{\Map(X,-)}"'] & \Spc 
		\end{tikzcd}
	\end{equation}
	is left adjointable.
	We claim that the induced square of presentable \categories
	\begin{equation}\tag*{$ (\textup{\ref{sq:SpcsquareSullivan}}) \tensor \Spc_{\leq 1} $}\label{sq:Spcsquaretrun}
		\begin{tikzcd}[column sep=5.5em, row sep=3em]
			\Spc_{\leq 1} \arrow[r, "{\Map(\trun_{\leq 1} X,-)}"] \arrow[d, "{\Map(\BCp,-)}"'] & \Spc_{\leq 1} \arrow[d, "{\Map(\BCp,-)}"] \\ 
			\Spc_{\leq 1} \arrow[r, "{\Map(\trun_{\leq 1} X,-)}"'] & \Spc_{\leq 1} 
		\end{tikzcd}
	\end{equation}
	is \textit{not} left adjointable.
	To see this note that the square \ref*{sq:Spcsquaretrun} is left adjointable if and only if the diagonal morphism
	\begin{equation*}
		\delta \colon \BCp \equivalent \trun_{\leq 1} X \longrightarrow \Map(\BCp, \trun_{\leq 1} X) \equivalent \Map(\BCp,\BCp) 
	\end{equation*}
	is an equivalence.
	However, the morphism $ \delta $ is not an equivalence: $ \uppi_0 \BCp \equivalent \ast $ and 
	\begin{equation*}
		\uppi_0\! \Map(\BCp,\BCp) \isomorphic \Hom_{\Ab}(\Cp,\Cp) \isomorphic \Cp \period
	\end{equation*}
\end{example}

When the functors $ \glowerstar $ and $ \gbarlowerstar $ are left adjoints, the requirement that the exchange transformation $ \fromto{\fupperstar \glowerstar}{\gbarlowerstar\fbarupperstar} $ be an equivalence \textit{is} expressible internally to $ \PrL $, hence is preserved by tensoring with any presentable \category:

\begin{lemma}\label{prop:extrartadj}
	Let $ E $ be a presentable \category.
	Assume that:
	\begin{enumerate}[label=\stlabel{prop:extrartadj}, ref=\arabic*]
		\item\label{prop:extrartadj.1} \eqref{sq:generalsquare} is an oriented square in $ \PrR $.

		\item\label{prop:extrartadj.2} The right adjoints $ \glowerstar \colon \fromto{C}{D} $ and $ \gbarlowerstar \colon \fromto{A}{B} $ admit right adjoints $ \guppersharp $ and $ \gbaruppersharp $, respectively.
		
		\item\label{prop:extrartadj.3} The oriented square \eqref{sq:generalsquare} is left adjointable.
	\end{enumerate}
	Then the functors $ \glowerstar \tensor E $ and $ \gbarlowerstar \tensor E $ are left adjoints and the oriented square $ \textup{\eqref{sq:generalsquare}} \tensor E $ is left adjointable.
\end{lemma}

\begin{proof}
	Assumption \enumref{prop:extrartadj}{2} implies and $ \glowerstar \tensor E $ and $ \gbarlowerstar \tensor E $ are left adjoint to $ \guppersharp \tensor E $ and $ \gbaruppersharp \tensor E $, respectively.
	Also note that by assumption $ \glowerstar $ and $ \gbarlowerstar $ admit right adjoints in the $ (\infty,2) $-category $ \PrR $.
	In light of \Cref{obs:adjointability}, assumption \enumref{prop:extrartadj}{3} implies that the square \eqref{sq:generalsquare} is vertically right adjointable in the $ (\infty,2) $-category $ \PrR $.
	Since functors of $ (\infty,2) $-categories preserve right adjointability, the square $ \textup{\eqref{sq:generalsquare}} \tensor E $ is vertically right adjointable in the $ (\infty,2) $-category $ \PrR $.
	Hence $ \textup{\eqref{sq:generalsquare}} \tensor E $ is vertically right adjointable in the $ (\infty,2) $-category $ \Cat_{\infty} $.
	Again applying \Cref{obs:adjointability}, we conclude that the square is horizontally left adjointable in the $ (\infty,2) $-category $ \Cat_{\infty} $.
\end{proof}

\section{Compactly generated \texorpdfstring{$ \infty $}{∞}-categories}\label{sec:CG}

In this section we recall a few facts about compactly generated \categories (\cref{subsec:CGreminder}) and give an explicit description of the tensor product with a compactly generated \category (\cref{subsec:tensorcg}).
We then give two useful applications of this description:
\begin{enumerate}
	\item In \cref{subsec:conservativity}, we show that many properties of a left adjoint between presentable \categories are preserved by tensoring with a compactly generated \category.

	\item In \cref{subsec:commuting_tensors_past_limits}, we show that tensoring with a compactly generated \category preserves limits of diagrams of presentable \categories and left exact left adjoints.

	\item In \cref{subsec:recollements}, we show that recollements of presentable \categories are preserved by tensoring with a compactly generated or stable presentable \category.
\end{enumerate} 


\subsection{Notations \& definitions}\label{subsec:CGreminder}

\begin{notation}
	Let $ E $ be \acategory with filtered colimits.
	We write $ \Ecpt \subset E $ for the full subcategory spanned by the compact objects.

	Recall that if $ E $ is compactly generated, then $ \Ecpt \subset E $ is closed under finite colimits and retracts.
	Moreover, $ E $ is the \defn{$ \Ind $-completion} of $ \Ecpt $.
	That is, $ E $ is obtained from $ \Ecpt $ by freely adjoining filtered colimits.
\end{notation}

We are also interested in an enlargement of the class of compactly generated \categories.

\begin{recollection}[(compactly assembled \categories)]\label{rec:compactly_assembled}
	A presentable \category $ E $ is \textit{compactly assembled} if $ E $ is a retract in $ \PrL $ of a compactly generated \category \cite[\SAGthm{Definition}{21.1.2.1} \& \SAGthm{Theorem}{21.1.2.18}]{SAG}.
	As a consequence of \SAG{Theorem}{21.1.2.10}, in a compactly assembled \category, filtered colimits are left exact.
\end{recollection}

\begin{nul}
	There are many important examples of presentable \categories which are compactly assembled but not compactly generated.
	For example, let $ M $ be a noncompact positive-dimensional topological manifold.
	Then the initial object is the only compact object of the \topos $ \Sh(M) $.
	(See \cite{MR1874232} for the stable variant of this statement.)
	However, the \topos of sheaves on a locally compact Hausdorff space is always compactly assembled \SAG{Proposition}{21.1.7.1}.
\end{nul}

\begin{recollection}[(projectively generated \categories)]\label{rec:compactproj}
	Let $ E $ be \acategory with sifted colimits, and let $ X \in E $.
	We say that $ X $ is \defn{projective} if $ \Map_E(X,-) \colon \fromto{E}{\Spc} $ preserves geometric realizations of simplicial objects.
	We say that $ X $ is \defn{compact projective} if $ X $ is compact and projective, i.e., $ \Map_E(X,-) \colon \fromto{E}{\Spc} $ preserves sifted colimits.
	We write $ \Ecptproj \subset E $ for the full subcategory spanned by the compact projective objects.

	We say that $ E $ is \defn{projectively generated} if there is a small collection of \textit{compact} projective objects of $ E $ that generate $ E $ under small colimits \HTT{Definition}{5.5.8.23}.
	In this case, $ \Ecptproj $ is closed under finite coproducts and retracts in $ E $.
	Moreover, $ E $ is the \defn{nonabelian derived category} of $ \Ecptproj $.
	That is, $ E $ is obtained from $ \Ecptproj $ by freely adjoining sifted colimits.
	See \cite[Propositions \HTTthmlink{5.5.8.15} \& \HTTthmlink{5.5.8.25}]{HTT}. 
\end{recollection}

\begin{example}
	The following \categories are projectively generated: the \category of spaces, the \category of connective modules over a connective $ \Eup_1 $-ring spectrum \HA{Proposition}{7.1.4.15}, the \category of animated (aka simplicial commutative) rings, and (up to set-theoretic issues) the \category of condensed/pyknotic spaces \cites[\S13.3]{exodromy}{pyknoticI}{Scholze:condensednotes}.
\end{example}

\begin{definition}
	We say that a presentable \category $ E $ is \textit{projectively assembled} if $ E $ is a retract in $ \PrL $ of a projectively generated \category.
\end{definition}


\subsection{Tensor products with compactly generated \texorpdfstring{$\infty$}{∞}-categories}\label{subsec:tensorcg}

Now we provide alternative models for tensor products with compactly generated \categories.
The key point is that these alternative models give us access to an explicit description of the action of the Lurie tensor product on a left adjoint functor.
These observations are known to experts (see \cites[\S2.3.1]{arXiv:1902.03404}[\S B.1]{arXiv:2012.10777}); we have included the material here because we were unable to locate a reference saying everything we need.

We begin with some terminology.

\begin{definition}\label{def:Klim}
	Let $ \Kcal $ be a collection of \categories.
	\begin{enumerate}[label=\stlabel{def:Klim}, ref=\arabic*]
		\item We say that \acategory $ I $ \defn{admits $ \Kcal $-shaped limits} if for each $ K \in \Kcal $, the \category $ I $ admits limits of $ K $-shaped diagrams.

		\item Given \categories $ I $ and $ S $ that admit $ \Kcal $-shaped limits, we say that a functor $ F \colon \fromto{I}{S} $ \defn{preserves limits of $ \Kcal $-shaped diagrams} if for each $ K \in \Kcal $, the functor $ F $ preserves limits of $ K $-shaped diagrams.

		\item Given \categories $ I $ and $ S $ that admit $ \Kcal $-shaped limits, we write $ \FunKlim(I,S) \subset \Fun(I,S) $ for the full subcategory spanned by those functors that preserve limits of $ \Kcal $-shaped diagrams.

		\item We write $ \CatKlim \subset \Cat_{\infty} $ for the (non-full) sub-$ (\infty,2) $-category of \categories admitting $ \Kcal $-shaped limits, functors preserving $ \Kcal $-shaped limits, and all natural transformations.

		\item If $ \Kcal $ is the collection of \textit{finite \categories}, we write $ \Funlex \colonequals \FunKlim $ and $ \Catlex \colonequals \CatKlim $.

		\item If $ \Kcal $ is the collection of \textit{finite sets}, we write $ \Funcross \colonequals \FunKlim $ and $ \Catfp \colonequals  \CatKlim $.
	\end{enumerate}
\end{definition}

\begin{observation}\label{obs:CGtensorFunlex}
	Let $ S $ and $ E $ be presentable \categories.
	\begin{enumerate}[label=\stlabel{obs:CGtensorFunlex}, ref=\arabic*]
		\item\label{obs:CGtensorFunlex.1} If $ E $ is compactly generated, then restriction along the inclusion $ \incto{\Ecptop}{E^{\op}} $ defines an equivalence of \categories
		\begin{equation*}\label{eq:tensorlex}
			\Funlim(E^{\op},S) \equivalence \Funlex(\Ecptop,S) \period
		\end{equation*}
		Hence the tensor product $ (-) \tensor E $ fits into a commutative square of functors of $ (\infty,2) $-categories
		\begin{equation*}
			\begin{tikzcd}[row sep=2em,column sep=6em]
				\PrR \arrow[d] \arrow[r, "{(-) \tensor E}"] & \PrR \arrow[d] \\
				\Catlex \arrow[r, "{\Funlex(\Ecptop,-)}"'] & \Catlex \period
			\end{tikzcd}
		\end{equation*}
		Here the vertical functors are inclusions of non-full subcategories.

		\item\label{obs:CGtensorFunlex.2} If $ E $ is projectively generated, then restriction along the inclusion $ \incto{\Ecptprojop}{E^{\op}} $ defines an equivalence of \categories
		\begin{equation*}\label{eq:tensorcross}
			\Funlim(E^{\op},S) \equivalence \Funcross(\Ecptprojop,S) \period
		\end{equation*}
		Hence the tensor product $ (-) \tensor E $ fits into a commutative square of functors of $ (\infty,2) $-categories
		\begin{equation*}
			\begin{tikzcd}[row sep=2em,column sep=6em]
				\PrR \arrow[d] \arrow[r, "{(-) \tensor E}"] & \PrR \arrow[d] \\
				\Catfp \arrow[r, "{\Funcross(\Ecptprojop,-)}"'] & \Catfp \period
			\end{tikzcd}
		\end{equation*}
		Here the vertical functors are inclusions of non-full subcategories.
	\end{enumerate}
\end{observation}

\begin{observation}\label{obs:tensorlex}
	Let $ E $ be a compactly generated \category and $ \pupperstar \colon \fromto{T}{S} $ be a \textit{left exact} left adjoint between presentable \categories with right adjoint $ \plowerstar $.
	Note that we have a commutative diagram of \categories
	\begin{equation*}
		\begin{tikzcd}
			S \tensor E \arrow[d, "\plowerstar \tensor E"'] \arrow[r, "\sim"{yshift=-0.2em}] & \Funlim(E^{\op},S) \arrow[d, "\plowerstar \of -"'] \arrow[r, "\sim"{yshift=-0.2em}] & \Funlex(\Ecptop,S)  \arrow[d, "\plowerstar \of -"] \\
			T \tensor E \arrow[r, "\sim"{yshift=-0.2em}] & \Funlim(E^{\op},T) \arrow[r, "\sim"{yshift=-0.2em}] & \Funlex(\Ecptop,T) \period
		\end{tikzcd}
	\end{equation*}
	Moreover, since $ \pupperstar $ is left exact, the functor
	\begin{equation*}
		\pupperstar \of - \colon \Funlex(\Ecptop,T) \to \Funlex(\Ecptop,S)
	\end{equation*}
	given by post-composition with $ \pupperstar $ is left adjoint to the functor given by post-composition with $ \plowerstar $.
	Hence we have a commutative square of \categories
	\begin{equation*}
		\begin{tikzcd}
			T \tensor E \arrow[d, "\pupperstar \tensor E"'] \arrow[r, "\sim"{yshift=-0.2em}] & \Funlex(\Ecptop,T)  \arrow[d, "\pupperstar \of -"] \\
			S \tensor E \arrow[r, "\sim"{yshift=-0.2em}] & \Funlex(\Ecptop,S) \period
		\end{tikzcd}
	\end{equation*}
\end{observation}

\begin{variant}\label{obs:tensorcross}
	Let $ E $ be a projectively generated \category and $ \pupperstar \colon \fromto{T}{S} $ be a left adjoint functor between presentable \categories that preserves finite products.
	Then we have a commutative square of \categories
	\begin{equation*}
		\begin{tikzcd}
			T \tensor E \arrow[d, "\pupperstar \tensor E"'] \arrow[r, "\sim"{yshift=-0.2em}] & \Funcross(\Ecptprojop,T)  \arrow[d, "\pupperstar \of -"] \\
			S \tensor E \arrow[r, "\sim"{yshift=-0.2em}] & \Funcross(\Ecptprojop,S) \period
		\end{tikzcd}
	\end{equation*}
\end{variant}

\Cref{obs:CGtensorFunlex,obs:tensorlex,obs:tensorcross} highlight that tensoring with a compactly or projectively generated \category has (unexpected) additional functoriality.


\subsection{Application: properties of left adjoints}\label{subsec:conservativity}

We now give two applications of \Cref{obs:tensorlex,obs:tensorcross} to the question of when tensoring with a presentable \category preserves the property of an adjoint being conservative or fully faithful.
First note that if $ \plowerstar $ is a conservative or fully faithful right adjoint, then \Cref{rec:HA.4.8.1.17} immediately implies that for any presentable \category $ E $, the functor $ \plowerstar \tensor E $ is conservative or fully faithful (see also \cite[Lemma 5.2.1]{arXiv:2007.13089v2}). 
The following example shows that the analogous claim for left adjoints is false:

\begin{example}[{\cite[Remark 5.2.2]{arXiv:2007.13089v2}}]
	Let $ \Spt $ denote the \category of spectra and $ \Spt_{\geq 0} \subset \Spt $ the full subcategory spanned by the connective spectra.
	The inclusion $ \Spt_{\geq 0} \subset \Spt $ admits a right adjoint given by taking the \textit{connective cover}.
	For any presentable $ 1 $-category $ E $, the tensor product $ \Spt \tensor E $ is the terminal category.
	In particular, $ \Spt \tensor \Set \equivalent \pt $.
	On the other hand, one can identify $ \Spt_{\geq 0} \tensor \Set $ with the category of abelian groups (as the heart of the standard \tstructure on $ \Spt $).
	Hence tensoring the inclusion $ \incto{\Spt_{\geq 0}}{\Spt} $ with $ \Set $ yields the zero functor $ \fromto{\Ab}{\pt} $; this functor is not conservative.
\end{example}

Note that the \categories $ \Spt_{\geq 0} $, $ \Spt $, and $ \Set $ are all compactly generated.
However, the inclusion $ \Spt_{\geq 0} \subset \Spt $ is not left exact, so we cannot apply \Cref{obs:tensorlex}.
This is the only obstruction to the preservation of conservativity or full faithfulness:

\begin{lemma}\label{lem:cgconservative}
	Let $ \{\pupperstar_i \colon \fromto{T}{S_i}\}_{i \in I} $ be a jointly conservative family of left adjoint functors between presentable \categories, and let $ E $ be a presentable \category.
	Assume that one of the following conditions holds:
	\begin{enumerate}[label=\stlabel{lem:cgconservative}, ref=\arabic*]
		\item\label{lem:cgconservative.1} The \category $ E $ is compactly assembled and the functors $ \{\pupperstar_i\}_{i \in I} $ are left exact.

		\item\label{lem:cgconservative.2} The \category $ E $ is projectively assembled and the functors $ \{\pupperstar_i\}_{i \in I} $ preserve finite products.
	\end{enumerate}
	Then the family of left adjoints $ \{\pupperstar_i \tensor E \colon \fromto{T \tensor E}{S_i \tensor E}\}_{i \in I} $ is jointly conservative.
\end{lemma}

\begin{proof}
	Since conservative functors are closed under retracts, by writing $ E $ as a retract in $ \PrL $ of a compactly or projectively assembled \category, it suffices to treat the cases where $ E $ is compactly or projectively assembled.
	By \Cref{obs:tensorlex}, in situation \enumref{lem:cgconservative}{1} it suffices to show that the collection of functors 
	\begin{equation*}
		\left\{\pupperstar_i \of - \colon \Funlex(\Ecptop,T) \to \Funlex(\Ecptop,S_i) \right\}_{i \in I}
	\end{equation*}
	is jointly conservative.
	Similarly, by \Cref{obs:tensorcross}, in situation \enumref{lem:cgconservative}{2} it suffices to show that the collection of functors 
	\begin{equation*}
		\left\{\pupperstar_i \of - \colon \Funcross(\Ecptprojop,T) \to \Funcross(\Ecptprojop,S_i) \right\}_{i \in I}
	\end{equation*}
	is jointly conservative.
	These assertions are immediate from the assumption that the functors $ \{\pupperstar_i\}_{i \in I} $ are jointly conservative.
\end{proof}

\begin{example}
	Given a jointly conservative family of points of \atopos, the family remains jointly conservative after tensoring with a compactly assembled \category.
\end{example}

Since fully faithful functors are closed under retracts, by the same style of argument we deduce:
 
\begin{lemma}\label{lem:cgff}
	Let $ \pupperstar \colon \incto{T}{S} $ be a fully faithful left adjoint functor between presentable \categories, and let $ E $ be a presentable \category.
	Assume that one of the following conditions holds:
	\begin{enumerate}[label=\stlabel{lem:cgff}, ref=\arabic*]
		\item\label{lem:cgff.1} The \category $ E $ is compactly assembled and $ \pupperstar $ is left exact.

		\item\label{lem:cgff.2} The \category $ E $ is projectively assembled and $ \pupperstar $ preserves finite products.
	\end{enumerate}
	Then the left adjoint $ \pupperstar \tensor E \colon \fromto{T \tensor E}{S \tensor E} $ fully faithful.
\end{lemma}



\subsection{Application: commuting tensors past limits}\label{subsec:commuting_tensors_past_limits}

Given a sheaf of presentable \categories on \asite, one is often interested in knowing if the sheaf condition is still satisfied after tensoring with another presentable \category.
Since the sheaf condition asks that certain diagrams be limit diagrams, it is useful to have an answer to the more general question of when tensoring with a presentable \category preserves limits in $ \PrL $.

In this subsection, we provide a useful situation in which tensoring with a compactly assembled \category commutes past limits in $ \PrL $.
As motivation, recall that a \textit{stable} presentable \category $ E $ is compactly assembled if and only if $ E $ is dualizable in the \category $ \PrLst $ of stable presentable \categories and left adjoints \SAG{Proposition}{D.7.3.1}.
Since the symmetric monoidal structure on $ \PrLst $ is closed, if $ E $ is dualizable, then it is immediate that $ E \tensor (-) \colon \fromto{\PrLst}{\PrLst} $ preserves limits.
The following is the unstable refinement of this fact:

\begin{lemma}\label{lem:tensoring_with_cg_preserves_limits}
	Let $ E $ be a presentable \category and  $ C_{\bullet} \colon \fromto{I^{\op}}{\PrL} $ a diagram.
	Assume one of the following:
	\begin{enumerate}[label=\stlabel{lem:tensoring_with_cg_preserves_limits}, ref=\arabic*]
		\item\label{lem:tensoring_with_cg_preserves_limits.1} The \category $ E $ is compactly assembled and for each morphism $ f \colon \fromto{i}{j} $ in $ I $, the induced functor $ \fupperstar \colon \fromto{C_j}{C_i} $ is left exact.

		\item\label{lem:tensoring_with_cg_preserves_limits.2} The \category $ E $ is projectively assembled and for each morphism $ f \colon \fromto{i}{j} $ in $ I $, the induced functor $ \fupperstar \colon \fromto{C_j}{C_i} $ preserves finite products.
	\end{enumerate}
	Then the natural left adjoint functor
	\begin{equation*}
		E \tensor \lim_{i \in I^{\op}} C_i \to \lim_{i \in I^{\op}} (E \tensor C_i)
	\end{equation*}
	is an equivalence.
	Here the limits are formed in $ \PrL $.
\end{lemma}

\begin{proof}	
	Since the proof is essential the same in both cases, we only prove \enumref{lem:tensoring_with_cg_preserves_limits}{1}.
	Since equivalences are closed under retracts, it suffices to treat the case where $ E $ is compactly \textit{generated}.
	Since the forgetful functors $ \fromto{\PrL}{\Catinfty} $ and $ \fromto{\Catlex}{\Catinfty} $ preserve limits and the composite
	\begin{equation*}
		\begin{tikzcd}
			I^{\op} \arrow[r, "C_{\bullet}"] & \PrL \arrow[r] & \Catinfty
		\end{tikzcd}
	\end{equation*}
	factors through $ \Catlex $, it suffices to prove the claim for limits computed in $ \Catlex $.
	Applying \Cref{obs:CGtensorFunlex,obs:tensorlex}, we see that
	\begin{align*}
		E \tensor \lim_{i \in I^{\op}} C_i &\equivalent \Funlex(\Ecptop,\textstyle\lim_{i \in I^{\op}} C_i) \\ 
		&\equivalence \lim_{i \in I^{\op}} \Funlex(\Ecptop,C_i) \\ 
		&\equivalent \lim_{i \in I^{\op}} (E \tensor C_i) \period \qedhere
	\end{align*}
\end{proof}

\begin{warning}
	In the statement of \Cref{lem:tensoring_with_cg_preserves_limits}, the assumption that the transition functors be left exact cannot generally be removed.
	For example, let $ E = \Spt $ be the \category of spectra and consider the limit diagram
	\begin{equation*}
		\Spc \equivalence \lim
		\left(
		\begin{tikzcd}
			\cdots \arrow[r, "\trun_{\leq n+1}"] & \Spc_{\leq n+1} \arrow[r, "\trun_{\leq n}"] & \Spc_{\leq n} \arrow[r, "\trun_{\leq n-1}"] & \cdots
		\end{tikzcd}
		\right) \period
	\end{equation*}
	We have $ \Spt \tensor \Spc \equivalent \Spt $. 
	On the other hand, $ \Spt \tensor \Spc_{\leq n} $ is the terminal category.
	Hence the limit $ \lim_{n \in \NN^{\op}} \Spt \tensor \Spc_{\leq n} $ is also the terminal category.
\end{warning}


\subsection{Application: recollements}\label{subsec:recollements}

Let $ X $ be \acategory with finite limits.
Recall that fully faithful functors
\begin{equation*}
	\ilowerstar \colon \incto{Z}{X} \andeq \jlowerstar \colon \incto{U}{X}
\end{equation*}
are said to exhibit $ X $ as the \textit{recollement} of $ Z $ and $ U $ if:%
\footnote{Here we use the convention for the open and closed pieces of a recollement from the theory of constructible sheaves.}
\begin{enumerate}
	\item The functors $ \ilowerstar $ and $ \jlowerstar $ admit left exact left adjoints $ \iupperstar $ and $ \jupperstar $, respectively.

	\item The functor $ \jupperstar \ilowerstar \colon \fromto{Z}{U} $ is constant at the terminal object of $ U $.

	\item The functors $ \iupperstar \colon \fromto{X}{Z} $ and $ \jupperstar \colon \fromto{X}{U} $ are jointly conservative.
\end{enumerate}
See \cites[\HAsec{A.8}]{HA}{Stablerecoll:BarwickGlasman}.
Primarily due to the requirement that $ \iupperstar $ and $ \jupperstar $ are jointly conservative, given a recollement of presentable \categories, it is not obvious if it remains a recollement after tensoring with another presentable \category.
We finish this section by showing that tensoring with a compactly generated or stable \category preserves recollements (\Cref{cor:prescgrecollement,prop:presstabrecollement}).

The following is immediate from the definitions.

\begin{proposition}\label{prop:Klimrecollement}
	Let $ \Kcal $ be a collection of \categories, let $ I $ be a small \category with $ \Kcal $-shaped limits, and let $ \ilowerstar \colon \incto{Z}{X} $ and $ \jlowerstar \colon \incto{U}{X} $ be fully faithful right adjoints between \categories that admit $ \Kcal $-shaped limits.
	Assume that $ \ilowerstar $ and $ \jlowerstar $ exhibit $ X $ as the recollement of $ Z $ and $ U $ and that the left adjoints $ \iupperstar $ and $ \jupperstar $ preserve $ \Kcal $-shaped limits.
	Then the functors 
	\begin{equation*}
		\ilowerstar \of - \colon \incto{\FunKlim(I,Z)}{\FunKlim(I,X)} 
		\andeq
		\jlowerstar \of - \colon \incto{\FunKlim(I,U)}{\FunKlim(I,X)}
	\end{equation*}
	exhibit $ \FunKlim(I,X) $ as the recollement of $ \FunKlim(I,Z) $ and $ \FunKlim(I,U) $.
\end{proposition}

\noindent The following consequence was previously recorded by Aizenbud and Carmeli \cite[Lemma 3.0.10]{arXiv:2105.12417v1}.

\begin{corollary}\label{cor:prescgrecollement}
	Let $ E $ be a presentable \category, and let $ \ilowerstar \colon \incto{Z}{X} $ and $ \jlowerstar \colon \incto{U}{X} $ be fully faithful right adjoints of presentable \categories that exhibit $ X $ as the recollement of $ Z $ and $ U $.
	If $ E $ is compactly generated, then $ \ilowerstar \tensor E $ and $ \jlowerstar \tensor E $ exhibit $ X \tensor E $ as the recollement of $ Z \tensor E $ and $ U \tensor E $.
\end{corollary}

\begin{proof}
	Combine \Cref{obs:CGtensorFunlex,obs:tensorlex} with \Cref{prop:Klimrecollement} in the case that $ \Kcal $ is the collection of finite \categories and $ I = \Ecptop $.
\end{proof}

Now we use \Cref{cor:prescgrecollement} and properties of recollements of stable \categories to show that tensoring with a presentable stable \category preserves recollements.

\begin{notation}
	We write $ \Spt $ for the \category of spectra.
	Recall that $ \Spt $ is compactly generated, and for any presentable stable \category $ E $, there is a natural equivalence $ \Omega_{E}^{\infty} \colon \equivto{\Spt \tensor E}{E} $ \cite[\HAthm{Proposition}{1.4.2.21} \& \HAthm{Example}{4.8.1.23}]{HA}.
\end{notation}

\begin{observation}\label{obs:stabadjoint}
	Let $ E $ be a presentable stable \category and $ \plowerstar \colon \fromto{S}{T} $ a right adjoint between presentable \categories.
	Since $ E $ is stable, $ \plowerstar \tensor \Spt \tensor E \equivalent \plowerstar \tensor E $.
	Thus, if $ \plowerstar \tensor \Spt $ admits a right adjoint, then $ \plowerstar \tensor E $ admits a right adjoint.
	Similarly, if $ \pupperstar \tensor \Spt $ admits a left adjoint, then $ \pupperstar \tensor E $ admits a left adjoint.
\end{observation}

\begin{recollection}\label{rec:stabrecollement}
	Let $ \ilowerstar \colon \incto{Z}{X} $ and $ \jlowerstar \colon \incto{U}{X} $ be functors that exhibit $ X $ as the \textit{recollement} of $ Z $ and $ U $.
	If the \category $ Z $ has an initial object, then $ \jupperstar $ admits a fully faithful left adjoint $ \jlowershriek \colon \incto{U}{X} $ \HA{Corollary}{A.8.13}.
	If, moreover, $ X $ has a zero object, then $ \ilowerstar $ admits a right adjoint $ \iuppershriek \colon \fromto{X}{Z} $ defined by taking the fiber
	\begin{equation*}
		\iuppershriek \colonequals \fib(\id{X} \to \jlowerstar\jupperstar)
	\end{equation*}
	of the unit $ \fromto{\id{X}}{\jlowerstar\jupperstar} $ \HAa{Remark}{A.8.5}.

	If $ X $ is stable, then $ Z $ and $ U $ are also stable.
	Moreover, there is a canonical fiber sequence
	\begin{equation}\label{eq:recfibseq}
		\begin{tikzcd}
			\jlowershriek \jupperstar \arrow[r] & \id{X} \arrow[r] & \ilowerstar \iupperstar \comma
		\end{tikzcd}
	\end{equation}
	where the first morphism is the counit and the second is the unit \cites[\HAappthm{Proposition}{A.8.17}]{HA}[1.17]{arXiv:1909.03920}.
\end{recollection}

Note that given the adjunctions of stable \categories $ \jlowershriek \leftadjoint \jupperstar $ and $ \iupperstar \leftadjoint \ilowerstar $, the existence of a fiber sequence \eqref{eq:recfibseq} implies that $ \iupperstar $ and $ \jupperstar $ are jointly conservative.
To show that tensoring with a presentable stable \category $ E $ preserves recollements, we prove that such a fiber sequence always exists by embedding $ E $ in a compactly generated stable \category.
We first check that the relevant adjoints exist.

\begin{notation}
	Let $ T $ be a presentable \category and let $ \ilowerstar \colon \incto{Z}{X} $ and $ \jlowerstar \colon \incto{U}{X} $ be fully faithful right adjoints of presentable \categories that exhibit $ X $ as the recollement of $ Z $ and $ U $.
	We write $ \iupperstar_T \colonequals \iupperstar \tensor T $ and $ \jupperstar_T \colonequals \jupperstar \tensor T $, and write $ \ilowerstar^T \colonequals \ilowerstar \tensor T $ and $ \jlowerstar^T \colonequals \jlowerstar \tensor T $.
	If $ \ilowerstar^T $ admits a right adjoint, we denote this adjoint by $ \iuppershriek_T $; if $ \jupperstar_T $ admits a left adjoint, we denote this adjoint by $ \jlowershriek^T $.
\end{notation}

\begin{lemma}\label{lem:ilowestarladj}
	Let $ E $ be a presentable \category, and let $ \ilowerstar \colon \incto{Z}{X} $ and $ \jlowerstar \colon \incto{U}{X} $ be fully faithful right adjoints of presentable \categories that exhibit $ X $ as the recollement of $ Z $ and $ U $.
	If $ E $ is stable, then: 
	\begin{enumerate}[label=\stlabel{lem:ilowestarladj}, ref=\arabic*]
		\item\label{lem:ilowestarladj.1} The functor $ \ilowerstar^E $ admits a right adjoint $ \iuppershriek_E $.

		\item\label{lem:ilowestarladj.2} The functor $ \jupperstar_E $ admits a left adjoint $ \jlowershriek^E $.
		
		\item\label{lem:ilowestarladj.3} The composite $ \jupperstar_E \ilowerstar^E \colon \fromto{Z \tensor E}{U \tensor E} $ is constant at the terminal object of $ U \tensor E $.
	\end{enumerate}
\end{lemma}

\begin{proof}
	By \Cref{cor:prescgrecollement}, $ \ilowerstar^{\Spt} $ and $ \jlowerstar^{\Spt} $ exhibit $ X \tensor \Spt $ as the recollement of $ Z \tensor \Spt $ and $ U \tensor \Spt $.
	Thus \enumref{lem:ilowestarladj}{1} and \enumref{lem:ilowestarladj}{2} follow from \Cref{obs:stabadjoint}.
	For \enumref{lem:ilowestarladj}{3}, note that since $ E $ is stable, we have a commuative diagram of right adjoints
	\begin{equation*}
		\begin{tikzcd}[column sep=5em]
			\Funlim(E^{\op},Z \tensor \Spt) \arrow[r, hooked, "\ilowerstar^{\Spt} \of -"] & \Funlim(E^{\op},X \tensor \Spt) \arrow[r, "\jupperstar_{\Spt} \of -"] & \Funlim(E^{\op},U \tensor \Spt)  \\ 
			Z \tensor \Spt \tensor E \arrow[u, "\wr"{xshift=0.2em}] \arrow[d, "\wr"'{xshift=0.2em}, "\Omega_{Z \tensor E}^{\infty}"] \arrow[r, hooked, "\ilowerstar^{\Spt} \tensor E" description] &
			X \tensor \Spt \tensor E \arrow[u, "\wr"{xshift=0.2em}] \arrow[d, "\wr"'{xshift=0.2em}, "\Omega_{X \tensor E}^{\infty}"] \arrow[r, "\jupperstar_{\Spt} \tensor E" description] &
			U \tensor \Spt \tensor E \arrow[u, "\wr"'{xshift=-0.2em}]  \arrow[d, "\wr"{xshift=-0.2em}, "\Omega_{U \tensor E}^{\infty}"'] \\
			Z \tensor E \arrow[r, hooked, "\ilowerstar^E"'] & X \tensor E \arrow[r, "\jupperstar_E"'] & U \tensor E  \period
		\end{tikzcd}
	\end{equation*}
	By \Cref{cor:prescgrecollement} the composite $ \jupperstar_{\Spt} \ilowerstar^{\Spt} $ is constant at the terminal object of $ U \tensor \Spt $, completing the proof.
\end{proof}

\begin{recollection}[{\HA{Proposition}{1.4.4.9}}]\label{rec:HA.1.4.4.9}
	\Acategory $ E $ is presentable and stable if and only if there exists a small \category $ E_0 $ such that $ E $ is equivalent to an accessible exact localization of $ \Fun(E_0,\Spt) $.
\end{recollection}

\begin{proposition}\label{prop:presstabrecollement}
	Let $ E $ be a presentable \category, and let $ \ilowerstar \colon \incto{Z}{X} $ and $ \jlowerstar \colon \incto{U}{X} $ be fully faithful right adjoints of presentable \categories that exhibit $ X $ as the recollement of $ Z $ and $ U $.
	If $ E $ is stable, then $ \ilowerstar \tensor E $ and $ \jlowerstar \tensor E $ exhibit $ X \tensor E $ as the recollement of $ Z \tensor E $ and $ U \tensor E $.
\end{proposition}

\begin{proof}
	Since $ X \tensor E $, $ Z \tensor E $, and $ U \tensor E $ are stable, the left adjoints $ \iupperstar_E $ and $ \jupperstar_E $ are exact.
	In light of \Cref{lem:ilowestarladj}, the remaining point to check is that the functors $ \iupperstar_E $ and $ \jupperstar_E $ are jointly conservative.
	To do this, use \Cref{rec:HA.1.4.4.9} to choose a compactly generated stable \category $ E' $ and fully faithful right adjoint $ \incto{E}{E'} $ with exact left adjoint $ L \colon \fromto{E'}{E} $.
	For a presentable \category $ T $, write $ L_T \colonequals T \tensor L $.

	By \Cref{cor:prescgrecollement}, $ \ilowerstar^{E'} $ and $ \jlowerstar^{E'} $ exhibit $ X \tensor E' $ as the recollement of $ Z \tensor E' $ and $ U \tensor E' $.
	Since $ E' $ is stable, there is a fiber sequence
	\begin{equation}\label{eq:fibseqE'}
		\begin{tikzcd}
			\jlowershriek^{E'} \jupperstar_{E'} \arrow[r] & \id{X \tensor E'} \arrow[r] & \ilowerstar^{E'} \iupperstar_{E'} 
		\end{tikzcd}
	\end{equation}
	of left adjoint functors.
	Applying \Cref{obs:adjointcompat,lem:ilowestarladj}, we see that 
	\begin{align*}
		L_X \jlowershriek^{E'} \jupperstar_{E'} &\equivalent \jlowershriek^{E} L_U \jupperstar_{E'} \equivalent \jlowershriek^{E} \jupperstar_{E} L_X \\ 
		\shortintertext{and}
		L_X \ilowerstar^{E'} \iupperstar_{E'} &\equivalent \ilowerstar^{E} L_Z \iupperstar_{E'} \equivalent \ilowerstar^{E} \iupperstar_{E} L_X \period
	\end{align*}
	Thus the fiber sequence \eqref{eq:fibseqE'} localizes to a fiber sequence of left adjoints
	\begin{equation}\label{eq:fibseqE}
		\begin{tikzcd}
			\jlowershriek^{E} \jupperstar_{E} \arrow[r] & \id{X \tensor E} \arrow[r] & \ilowerstar^E \iupperstar_{E} \period
		\end{tikzcd}
	\end{equation}
	To see that $ \iupperstar_E $ and $ \jupperstar_E $ are jointly conservative, note that if $ F \in X \tensor E $ and $ \jupperstar_E(F) = 0 $ and $ \iupperstar_E(F) = 0 $, then the fiber sequence \eqref{eq:fibseqE} shows that $ F = 0 $.
\end{proof} 

\begin{corollary}\label{cor:presstabrecollement}
	Let $ E $ be a presentable \category, and let $ \ilowerstar \colon \incto{Z}{X} $ and $ \jlowerstar \colon \incto{U}{X} $ be fully faithful right adjoints of presentable \categories that exhibit $ X $ as the recollement of $ Z $ and $ U $.
	If $ X $ is stable, then $ \ilowerstar \tensor E $ and $ \jlowerstar \tensor E $ exhibit $ X \tensor E $ as the recollement of $ Z \tensor E $ and $ U \tensor E $.
\end{corollary}

\begin{proof}
	Since $ X $ is stable, both $ Z $ and $ U $ are stable (\Cref{rec:stabrecollement}) and we have a commutative diagram
	\begin{equation*}
		\begin{tikzcd}[column sep=5em]
			Z \tensor \Spt \arrow[d, "\wr"'{xshift=0.2em}, "\Omega_Z^{\infty}"] \arrow[r, hooked, "\ilowerstar \tensor \Sp"] & X \tensor \Spt \arrow[d, "\wr"'{xshift=0.2em}, "\Omega_X^{\infty}"] & U \tensor \Spt  \arrow[d, "\wr"{xshift=-0.2em}, "\Omega_U^{\infty}"'] \arrow[l, hooked', "\jlowerstar \tensor \Sp"'] \\
			Z \arrow[r, hooked, "\ilowerstar"'] & X & U \arrow[l, hooked', "\jlowerstar"] \period
		\end{tikzcd}
	\end{equation*}
	Hence the claim is equivalent to showing that $ \ilowerstar \tensor (\Spt \tensor E) $ and $ \jlowerstar \tensor (\Spt \tensor E) $ exhibit $ X \tensor (\Spt \tensor E) $ as the recollement of $ Z \tensor (\Spt \tensor E) $ and $ U \tensor (\Spt \tensor E) $.
	Since $ \Spt \tensor E $ is stable, \Cref{prop:presstabrecollement} completes the proof.
\end{proof}

\begin{remark}
	Contemporaneously with the first version of this work, Carmeli, Schlank, and Yanovski \cite[Proposition 5.2.3]{arXiv:2007.13089v2} provided a different proof of \Cref{cor:presstabrecollement}.
\end{remark}

\begin{remark}
	If $ \ilowerstar \colon \incto{Z}{X} $ and $ \jlowerstar \colon \incto{U}{X} $ form a recollement of presentable \categories, and $ X $ is \atopos, then $ Z $ and $ U $ are \topoi \HAa{Proposition}{A.8.15}.
	Moreover, if $ E $ is another \topos, then $ \ilowerstar \tensor E $ and $ \jlowerstar \tensor E $ exhibit $ X \tensor E $ as the recollement of $ Z \tensor E $ and $ U \tensor E $ \cites[\HTTthm{Remark}{6.3.5.8} \& \HTTthm{Proposition}{7.3.2.12}]{HTT}[\HAthm{Example}{4.8.1.19} \& \HAappthm{Proposition}{A.8.15}]{HA}.
	In light of this and \Cref{cor:prescgrecollement,prop:presstabrecollement,cor:presstabrecollement}, in basically all situations one naturally runs into, the tensor product preserves recollements.
\end{remark}


\section{Adjointability results}\label{sec:results}

In this section, we use the explicit descriptions of the tensor product with a compactly generated \category from \cref{subsec:tensorcg} to explain which operations on an oriented square \eqref{sq:generalsquare} of presentable \categories preserve left adjointability.
In particular, we prove \Cref{thm:main}.

In \cref{sec:functors}, we make a general observation (\Cref{prop:Klimadj}) that immediately takes care of the compactly generated case of \Cref{thm:main}.
\Cref{prop:Klimadj} also has some other useful consequences; see \Cref{ex:Lawvere,cor:CGprojadjointability}.
In \cref{sec:filteredcolim}, we take care of the stable case of \Cref{thm:main}.
In \cref{subsec:properbasechange}, we explain consequences of \Cref{thm:main} and Lurie's Nonabelian Proper Basechange Theorem. 


\subsection{Adjointability \& \texorpdfstring{$\infty$}{∞}-categories of functors}\label{sec:functors}

To prove the compactly generated case of \Cref{thm:main}, we appeal to the improved functoriality of the tensor product explained in \Cref{obs:CGtensorFunlex}.

\begin{proposition}\label{prop:Klimadj}
	Let $ \Kcal $ be a collection of \categories and let $ I $ be a small \category.
	Assume that:
	\begin{enumerate}[label=\stlabel{prop:Klimadj}, ref=\arabic*]
		\item The \category $ I $ admits $ \Kcal $-shaped limits and \eqref{sq:generalsquare} is an oriented square in $ \CatKlim $.

		\item The left adjoints $ \fupperstar \colon \fromto{D}{B} $ and $ \fbarupperstar \colon \fromto{C}{A} $ preserve $ \Kcal $-shaped limits.
		
		\item The oriented square \eqref{sq:generalsquare} is left adjointable.
	\end{enumerate}
	Then the induced oriented square
	\begin{equation}\tag*{$ \FunKlim(I,(\textup{\ref{sq:generalsquare}})) $}\label{sq:FunKlimgeneralsquare}
		\begin{tikzcd}[column sep=3.5em, row sep=2em]
			\FunKlim(I,A) \arrow[r, "\fbarlowerstar \of -"] \arrow[d, "\gbarlowerstar \of -"'] & \FunKlim(I,C) \arrow[d, "\glowerstar \of -"] \arrow[dl, phantom, "\scriptstyle \sigma \of -" below right, "\Longleftarrow" sloped] \\ 
			\FunKlim(I,B) \arrow[r, "\flowerstar \of -"'] & \FunKlim(I,D) 
		\end{tikzcd}
	\end{equation}
	is left adjointable.
\end{proposition}

\begin{proof}
	By the assumptions, the functors $ \flowerstar $ and $ \fbarlowerstar $ admit left adjoints in the $ (\infty,2) $-category $ \CatKlim $.
	The fact that functors of $ (\infty,2) $-categories preserves left adjointable squares completes the proof.
\end{proof}


\begin{corollary}\label{cor:CGadjointability}
	Let $ E $ be a compactly generated \category.
	Assume that:
	\begin{enumerate}[label=\stlabel{cor:CGadjointability}, ref=\arabic*]
		\item\label{cor:CGadjointability.1} \eqref{sq:generalsquare} is an oriented square in $ \PrR $.

		\item\label{cor:CGadjointability.2} The left adjoints $ \fupperstar \colon \fromto{D}{B} $ and $ \fbarupperstar \colon \fromto{C}{A} $ are left exact.
		
		\item\label{cor:CGadjointability.3} The oriented square \eqref{sq:generalsquare} is left adjointable.
	\end{enumerate}
	Then the oriented square $ \textup{\eqref{sq:generalsquare}} \tensor E $ is left adjointable.
\end{corollary}

\begin{proof}
	Combine \Cref{obs:CGtensorFunlex,obs:tensorlex} with \Cref{prop:Klimadj} in the case that $ \Kcal $ is the collection of finite \categories and $ I = \Ecptop $.
\end{proof}

\begin{warning}
	Note that the \category $ \Spc_{\leq 1} $ of $ 1 $-truncated spaces is compactly generated.
	Hence \Cref{ex:trunSpcsquare} shows that the assumption \enumref{cor:CGadjointability}{2} cannot be removed.
\end{warning}

\begin{corollary}\label{cor:productadjointability}
	Let $ I $ be a small \category with finite products.
	Assume that:
	\begin{enumerate}[label=\stlabel{cor:productadjointability}, ref=\arabic*]
		\item \eqref{sq:generalsquare} is an oriented square in $ \Catfp $.

		\item The left adjoints $ \fupperstar \colon \fromto{D}{B} $ and $ \fbarupperstar \colon \fromto{C}{A} $ preserve finite products.
		
		\item The oriented square \eqref{sq:generalsquare} is left adjointable.
	\end{enumerate}
	Then the oriented square $ \Funcross(I,\textup{\eqref{sq:generalsquare}}) $ is left adjointable.
\end{corollary}

\begin{proof}
	Combine \Cref{obs:CGtensorFunlex,obs:tensorcross} with \Cref{prop:Klimadj} in the case that $ \Kcal $ is the collection of finite sets.
\end{proof}

\Cref{cor:productadjointability} has some nice consequences:

\begin{example}[(algebras over Lawvere theories)]\label{ex:Lawvere}
	Let $ L $ be a Lawvere theory in the \categorical sense \cites{MR3848405}{MR4083788}[Chapter 3]{arXiv:1011.3243}.
	In the setting of \Cref{cor:productadjointability}, the induced square of \categories of $ L $-algebras
	\begin{equation*}
		\begin{tikzcd}
			\Alg_L(A) \arrow[r] \arrow[d] & \Alg_L(C) \arrow[d] \arrow[dl, phantom, "\Longleftarrow" sloped] \\ 
			\Alg_L(B) \arrow[r] & \Alg_L(D) \comma
		\end{tikzcd}
	\end{equation*}
	is left adjointable. 
	In particular, letting $ L = \Span(\Setfin) $ be the $ (2,1) $-category of spans of finite sets \cite[\S3]{MR3558219}, we see that the formation of commutative monoid objects preserves left adjointability of oriented squares in which all functors preserve finite products.
\end{example}

\noindent Another special case of \Cref{cor:productadjointability} is given by tensoring with a projectively generated \category (here take $ I = \Ecptprojop $):

\begin{corollary}\label{cor:CGprojadjointability}
	Let $ E $ be a projectively generated \category.
	Assume that:
	\begin{enumerate}[label=\stlabel{cor:CGprojadjointability}, ref=\arabic*]
		\item \eqref{sq:generalsquare} is an oriented square in $ \PrR $.

		\item The left adjoints $ \fupperstar \colon \fromto{D}{B} $ and $ \fbarupperstar \colon \fromto{C}{A} $ preserve finite products.
		
		\item The oriented square \eqref{sq:generalsquare} is left adjointable.
	\end{enumerate}
	Then the oriented square $ \textup{\eqref{sq:generalsquare}} \tensor E $ is left adjointable.
\end{corollary}


\subsection{Adjointability \& preservation of filtered colimits}\label{sec:filteredcolim}

In many \categories that arise in algebra and sheaf theory, filtered colimits commute with finite limits.
For example, filtered colimits commute with finite limits in: 
compactly generated \categories, Grothendieck abelian categories \cites[\stackstag{079A}]{stacksproject}[\HAthm{Definition}{1.3.5.1}]{HA}{MR102537}, Grothendieck prestable \categories \SAG{Definition}{C.1.4.2}, stable \categories, and $ n $-topoi for each $ 0 \leq  n \leq  \infty $.
The goal of this subsection is to show that in these situations, if the vertical right adjoints $ \glowerstar $ and $ \gbarlowerstar $ preserve filtered colimits and \eqref{sq:generalsquare} becomes left adjointable after tensoring with the \category of spectra, then \eqref{sq:generalsquare} becomes left adjointable after tensoring with \textit{any} stable presentable \category.
Combined with \Cref{cor:CGadjointability}, this allows us to generalize the Proper Basechange Theorem in topology to sheaves with values in presentable \categories which are compactly generated or stable (\Cref{subex:propertopology}).

In order to state this result, we recall some terminology.

\begin{recollection}
	Let $ S $ be \acategory with finite limits and filtered colimits.
	We say that \defn{filtered colimits in $ S $ are left exact} if for each small filtered \category $ I $, the functor
	\begin{equation*}
		\textstyle \colim_I \colon \fromto{\Fun(I,S)}{S}
	\end{equation*}
	is left exact.
\end{recollection}

\begin{proposition}\label{cor:stabtensorfilteredcolim}
	Let $ E $ be a stable presentable \category.
	Assume that:
	\begin{enumerate}[label=\stlabel{cor:stabtensorfilteredcolim}, ref=\arabic*]
		\item\label{cor:stabtensorfilteredcolim.1} The \categories $ A $, $ B $, $ C $, and $ D $ in the oriented square \eqref{sq:generalsquare} are presentable and filtered colimits are left exact in each \category.
		Moreover, all functors in \eqref{sq:generalsquare} are right adjoints.

		\item\label{cor:stabtensorfilteredcolim.2} The right adjoints $ \glowerstar $ and $ \gbarlowerstar $ preserve filtered colimits.
		
		\item\label{cor:stabtensorfilteredcolim.3} The oriented square $ \textup{\eqref{sq:generalsquare}} \tensor \Spt $ is left adjointable.
	\end{enumerate}
	Then the functors $ \glowerstar \tensor E $ and $ \gbarlowerstar \tensor E $ are left adjoints and the oriented square $ \textup{\eqref{sq:generalsquare}} \tensor E $ is left adjointable.
\end{proposition}

\begin{remark}
	Since the \category $ \Spt $ is compactly generated, \Cref{cor:CGadjointability} shows that if the oriented square \eqref{sq:generalsquare} is left adjointable and the left adjoints $ \fupperstar $ and $ \fbarupperstar $ are left exact, then \enumref{cor:stabtensorfilteredcolim}{3} is satisfied.
	In particular, hypotheses \enumref{cor:stabtensorfilteredcolim}{1} and \enumref{cor:stabtensorfilteredcolim}{3} are satisfied for left adjointable squares of \topoi and geometric morphisms.
\end{remark}

To prove \Cref{cor:stabtensorfilteredcolim}, we begin with a few basic lemmas.
The key point is that the assumption that $ \glowerstar $ and $ \gbarlowerstar $ preserve filtered colimits implies that $ \glowerstar \tensor \Spt $ and $ \gbarlowerstar \tensor \Spt $ are left adjoints (\Cref{cor:filteredcolimextrartadj}).
Thus we are in the situation to apply \Cref{prop:extrartadj}.

The following is immediate from the definitions.

\begin{lemma}\label{lem:Funlexclosed}
	Let $ I $ be \acategory with finite limits and let $ S $ be \acategory with finite limits and filtered colimits.
	If filtered colimits in $ S $ are left exact, then $ \Funlex(I,S) \subset \Fun(I,S) $ is closed under filtered colimits.
\end{lemma}

\begin{corollary}\label{cor:tensorpreservefilteredcolim}
	Let $ \plowerstar \colon \fromto{S}{T} $ be a right adjoint between presentable \categories in which filtered colimits are left exact.
	Let $ E $ be a compactly generated \category.
	If $ \plowerstar $ preserves filtered colimits, then $ \plowerstar \tensor E $ preserves filtered colimits.
\end{corollary}

\begin{proof}
	Consider the commutative diagram of \categories
	\begin{equation*}
		\begin{tikzcd}[column sep=4.5em, row sep=2em]
			S \tensor E \arrow[d, "\plowerstar \tensor E"'] \arrow[r, "\sim"{yshift=-0.2em}] & \Funlex(E^{\op},S) \arrow[d, "\plowerstar \of -"'] \arrow[r, hooked] & \Fun(\Ecptop,S)  \arrow[d, "\plowerstar \of -"] \\
			T \tensor E \arrow[r, "\sim"{yshift=-0.2em}] & \Funlex(E^{\op},T) \arrow[r, hooked] & \Fun(\Ecptop,T) \period
		\end{tikzcd}
	\end{equation*}
	Since $ \plowerstar $ preserves filtered colimits, the rightmost vertical functor preserves filtered colimits. 
	The claim now follows from \Cref{lem:Funlexclosed}. 
\end{proof}

\begin{corollary}\label{cor:filteredcolimextrartadj}
	Let $ \plowerstar \colon \fromto{S}{T} $ be a right adjoint between presentable \categories in which filtered colimits are left exact.
	Let $ E $ be a stable presentable \category.
	If $ \plowerstar $ preserves filtered colimits, then the right adjoint functor $ \plowerstar \tensor E $ is also a left adjoint.
\end{corollary}

\begin{proof}
	By \Cref{obs:stabadjoint}, it suffices to show that $ \plowerstar \tensor \Spt $ is a left adjoint.
	Since $ S \tensor \Spt $ and $ T \tensor \Spt $ are stable and $ \plowerstar \tensor \Spt $ is exact, by the Adjoint Functor Theorem, it suffices to show that $ \plowerstar \tensor \Spt $ preserves filtered colimits.
 	\Cref{cor:tensorpreservefilteredcolim} completes the proof.
\end{proof}

\begin{proof}[Proof of \Cref{cor:stabtensorfilteredcolim}]
	\Cref{cor:filteredcolimextrartadj} shows that $ \glowerstar \tensor \Spt $ and $ \gbarlowerstar \tensor \Spt $ are left adjoints.
	Since $ E $ is stable, $ \textup{\eqref{sq:generalsquare}} \tensor \Spt \tensor E \equivalent \textup{\eqref{sq:generalsquare}} \tensor E $; applying \Cref{prop:extrartadj} to the oriented square $ \textup{\eqref{sq:generalsquare}} \tensor \Spt $ completes the proof.
\end{proof}


\subsection{Consequences of the Nonabelian Proper Basechange Theorem}\label{subsec:properbasechange}

We finish by explaining how \Cref{cor:CGadjointability,cor:stabtensorfilteredcolim} answer \Cref{quest:motivating}.

\begin{example}\label{ex:proper}
	Let
	\begin{equation}\label{sq:proper}
		\begin{tikzcd}
			\W \arrow[r, "\fbarlowerstar"] \arrow[d, "\gbarlowerstar"'] \arrow[dr, phantom, very near start, "\lrcorner", xshift=-0.25em, yshift=0.25em] & \Y \arrow[d, "\glowerstar"] \\ 
			\X \arrow[r, "\flowerstar"'] & \Z 
		\end{tikzcd}
	\end{equation}
	be a pullback square in the \category of \topoi and (right adjoints in) geometric morphisms.
	Assume that the geometric morphism $ \glowerstar $ is \textit{proper} in the sense of \HTT{Definition}{7.3.1.4}.
	Then the square \eqref{sq:proper} is left adjointable, the geometric morphism $ \gbarlowerstar $ is also proper, and the functors $ \glowerstar $ and $ \gbarlowerstar $ preserve filtered colimits \HTT{Remark}{7.3.1.5}.
	Applying \Cref{cor:CGadjointability,cor:stabtensorfilteredcolim}, we see that if $ E $ is a presentable \category which is stable or compactly generated, then the square $ \textup{\eqref{sq:proper}} \tensor E $ is left adjointable.
\end{example}

\begin{subexample}\label{subex:propertopology}
	Let
	\begin{equation*}
		\begin{tikzcd}
			W \arrow[r, "\fbar"] \arrow[d, "\gbar"'] \arrow[dr, phantom, very near start, "\lrcorner", xshift=-0.25em, yshift=0.25em] & Y \arrow[d, "g"] \\ 
			X \arrow[r, "f"'] & Z 
		\end{tikzcd}
	\end{equation*}
	be a pullback square of locally compact Hausdorff topological spaces, and assume that the map $ g $ is proper.
	By \HTT{Theorem}{7.3.1.16}, the geometric morphism $ \glowerstar \colon \fromto{\Sh(Y;\Spc)}{\Sh(Z;\Spc)} $ is proper.
	As a special case of \Cref{ex:proper}, we see that if $ E $ is a presentable \category which is stable or compactly generated, then the induced square of \categories of $ E $-valued sheaves 
	\begin{equation*}
		\begin{tikzcd}
			\Sh(W;E) \arrow[r, "\fbarlowerstar"] \arrow[d, "\gbarlowerstar"'] & \Sh(Y;E) \arrow[d, "\glowerstar"] \\ 
			\Sh(X;E) \arrow[r, "\flowerstar"'] & \Sh(Z;E) 
		\end{tikzcd}
	\end{equation*}
	is left adjointable.
\end{subexample}


\section{Adjointability and stabilization}\label{sec:adjointability_stabilization}

The following situation commonly arises in sheaf theory: we often only know that the exchange morphism
\begin{equation*}
	\Ex_{\sigma} \colon \fupperstar \glowerstar \to \gbarlowerstar\fbarupperstar
\end{equation*}
associated to an oriented square \eqref{sq:generalsquare} of some \categories of sheaves is an equivalence when restricted to a (not necessarily presentable) subcategory $ C' \subset C $.
Such is the case for the Proper Base\-change Theorem for étale cohomology: the relevant exchange morphism is an equivalence for torsion sheaves, but fails to be an equivalence in general \cite[Exposé XII, \S2]{MR50:7132}.
In these situations, the adjointability results proven in the previous sections do not immediately allow one to conclude that basechange results for a class of sheaves spaces imply basechange results with other coefficients.

The purpose of this section is to explain how use knowledge that the exchange morphism associated to an oriented square is an equivalence when restricted to a subcategory to deduce adjointability results after stabilization or tensoring with the \category of modules over an $ \Eup_1 $-ring.
See \Cref{prop:stable_adjointability,cor:R-module_adjointability}.
The results of this section generalize our work with Barwick and Glasman \cite[\S7.4]{exodromy}.

We begin with the case of stabilization. 
For convenience, we work in the more general setting of \categories with finite limits.
First we recall how stabilization works in this setting.

\begin{recollection}[(stabilization \HA{Definition}{1.4.2.8})]
	Write $ \Spcfin \subset \Spc $ for the \category of \defn{finite spaces}: the smallest full subcategory of $ \Spc $ containing the terminal object and closed under finite colimits.
	Let $ S $ be \acategory with finite limits.
	Recall that the \defn{stabilization} of $ S $ is the full subcategory
	\begin{equation*}
		\Stab(S) \subset \Fun(\Spcfinpt,S)
	\end{equation*}
	spanned by those functors that preserve the terminal object and carry pushout squares in $ \Spcfinpt $ to pullback squares in $ S $.
	Also recall that the functor $ \Omega_S^{\infty} \colon \fromto{\Stab(S)}{S} $ is defined by evaluation on the $ 0 $-sphere \smash{$ \Sph{0} \in \Spcfinpt $}.
	Stabilization defines a functor of $ (\infty,2) $-categories
	\begin{equation*}
		\Stab \colon \fromto{\Catlex}{\Catlex} \semicolon 
	\end{equation*}
	it is a subfunctor of the functor $ \Fun(\Spcfinpt,-) $.

	If $ S $ is a presentable \category, then the stabilization $ \Stab(S) $ has another description: there is a natural equivalence
	\begin{equation*}
		\Stab(S) \equivalent S \tensor \Spt
	\end{equation*}
	\HA{Example}{4.8.1.23}.
	Similarly to \enumref{obs:CGtensorFunlex}{1}, the tensor product $ (-) \tensor \Spt $ fits into a commutative square of functors of $ (\infty,2) $-categories
	\begin{equation*}
		\begin{tikzcd}[row sep=2em,column sep=4em]
			\PrR \arrow[d] \arrow[r, "{(-) \tensor \Spt}"] & \PrR \arrow[d] \\
			\Catlex \arrow[r, "\Stab"'] & \Catlex \period
		\end{tikzcd}
	\end{equation*}
	Here the vertical functors are inclusions of non-full subcategories.
\end{recollection}

Stabilization behaves well with respect to the functors $ \Omega^{\infty} $ and exchange morphisms:

\begin{observation}\label{nul:Omegainftycommute}
	Let $ p \colon \fromto{S}{T} $ be a left exact functor between \categories with finite limits.
	It is immediate from the definitions that the square
	\begin{equation*}
      \begin{tikzcd}[sep=2.25em]
	       \Stab(S) \arrow[d, "\Omega_S^{\infty}"'] \arrow[r, "\Stab(p)"] & \Stab(T) \arrow[d, "\Omega_T^{\infty}"] \\ 
	       S \arrow[r, "p"'] & T 
      \end{tikzcd}
    \end{equation*}
    canonically commutes.
\end{observation}

\begin{observation}[(stabilization and naural transformations)]\label{nul:natonStab}
	Let $ p, p' \colon \fromto{S}{T} $ be left exact functors between \categories with finite limits, and let $ \sigma \colon \fromto{p}{p'} $ be a natural transformation.
	It is immediate from the definitions that the natural transformation $ \Stab(\sigma) $ is compatible with $ \sigma $ in the following sense: we have a natural identification $ \Omega_T^{\infty} \Stab(\sigma) = \sigma \Omega_S^{\infty} $ of natural tranformations
	\begin{equation*}
		\begin{tikzcd}[sep=2.5em]
			p \Omega_S^{\infty} = \Omega_T^{\infty} \Stab(p) \arrow[r] & \Omega_T^{\infty} \Stab(p') = p' \Omega_S^{\infty} \period
		\end{tikzcd}
	\end{equation*}
\end{observation}

\begin{observation}[(stabilization and exchange morphisms)]\label{nul:BCcompatStab}
	Consider an oriented square  \eqref{sq:generalsquare} in \smash{$ \Catlex $} and assume that the left adjoints $ \fupperstar \colon \fromto{D}{B} $ and $ \fbarupperstar \colon \fromto{C}{A} $ are left exact.
	From \Cref{nul:natonStab} we see that there is a natural identification
	\begin{equation*}
		\Omega_B^{\infty} \Ex_{\Stab(\sigma)} = \Ex_{\sigma} \Omega_C^{\infty} 
	\end{equation*}
	of natural transformations
	\begin{equation*}
		\begin{tikzcd}[sep=2.5em]
			\fupperstar\glowerstar \Omega_C^{\infty} = \Omega_B^{\infty} \Stab(\fupperstar) \Stab(\glowerstar) \arrow[r] & \Omega_B^{\infty} \Stab(\gbarlowerstar) \Stab(\gbarupperstar) = \gbarlowerstar \gbarupperstar \Omega_C^{\infty} \period
		\end{tikzcd}
	\end{equation*}
\end{observation}

In order to state the main result, it is convenient to give a name to the largest subcategory on which the exchange morphism is an equivalence.

\begin{notation}
	Given an oriented square \eqref{sq:generalsquare}, we write $ C_{\Ex} \subset C $ for the full subcategory spanned by those objects $ X \in C $ such that the exchange morphism
	\begin{equation*}
		\Ex_{\sigma}(X) \colon \fupperstar \glowerstar(X) \to \gbarlowerstar\fbarupperstar(X)
	\end{equation*}
	is an equivalence.
\end{notation}

\begin{observation}
	In the setting of \cref{nul:BCcompatStab}, the subcategory $ C_{\Ex} \subset C $ is closed under finite limits.
	In this case, the stabilization $ \Stab(C_{\Ex}) $ of $ C_{\Ex} $ is the full subcategory spanned by those $ X \in \Stab(C) $ such that for each $ n \in \ZZ $, we have $ \Omega_C^{\infty-n}(X) \in C_{\Ex} $.
\end{observation}

\begin{proposition}\label{prop:stable_adjointability}
	Assume that:
	\begin{enumerate}[label=\stlabel{prop:stable_adjointability}, ref=\arabic*]
		\item\label{prop:stable_adjointability.1} \eqref{sq:generalsquare} is an oriented square in $ \Catlex $.

		\item\label{prop:stable_adjointability.2} The left adjoints $ \fupperstar \colon \fromto{D}{B} $ and $ \fbarupperstar \colon \fromto{C}{A} $ are left exact.
	\end{enumerate}
	Then the exchange morphism associated to the oriented square of stable \categories
	\begin{equation*}\tag*{$ \Stab(\textup{\ref{sq:generalsquare}}) $}\label{sq:Stab_square}
		\begin{tikzcd}[column sep=2.75em, row sep=2em]
			\Stab(A) \arrow[r, "\Stab(\fbarlowerstar)" above] \arrow[d, "\Stab(\gbarlowerstar)" left] & \Stab(D) \arrow[d, "\Stab(\glowerstar)" right] \arrow[dl, phantom, "\Longleftarrow" sloped] \\ 
			\Stab(B) \arrow[r, "\Stab(\flowerstar)" below] & \Stab(C) \period
		\end{tikzcd}
	\end{equation*}
	is an equivalence when restricted to $ \Stab(C_{\Ex}) $.
\end{proposition}

\begin{proof}
	Let $ X \in \Stab(C_{\Ex}) $.
	To see that $ \Ex(X) $ is an equivalence, it suffices to show that for each integer $ n \in \ZZ $, the morphism
	\begin{equation*}
		\Omega_{B}^{\infty-n}\Ex(X) \colon \fromto{\Omega_{B}^{\infty-n} \fupperstar \glowerstar(X)}{\Omega_{B}^{\infty-n} \gbarlowerstar\fbarupperstar(X)} 
	\end{equation*}
	is an equivalence.
	Since all functors in question are left exact, applying \Cref{nul:BCcompatStab} we see that the morphism $ \Omega_{B}^{\infty-n}\Ex(X) $ is equivalent to the morphism
	\begin{equation*}
		\Ex(\Omega_{C}^{\infty-n}X) \colon \fromto{\fupperstar \glowerstar(\Omega_{C}^{\infty-n}X)}{\gbarlowerstar\fbarupperstar(\Omega_{C}^{\infty-n}X)} \period
	\end{equation*}
	The assumption that $ X \in \Stab(C_{\Ex}) $ guarantees that for all integers $ n \in \ZZ $, we have $ \Omega_{C}^{\infty-n} X \in C_{\Ex} $.
\end{proof}


\subsection{Adjointability and \texorpdfstring{$R$}{R}-modules}\label{subsec:adjointability_R-mods}

Now we explain how to bootstrap from \Cref{prop:stable_adjointability} to deduce analagous results when tensoring with the \category of modules over an $ \Eup_1 $-ring.

\begin{notation}\label{ntn:tensor_ModR}
	Let $ S $ be a presentable \category and $ R $ an $ \Eup_1 $-ring spectrum.
	\begin{enumerate}[label=\stlabel{ntn:tensor_ModR}]
		\item We write $ \Mod(R) $ for the \category of left $ R $-module spectra and $ \Uup \colon \fromto{\Mod(R)}{\Spt} $ for the forgetful functor.
		Note that $ \Uup $ is conservative as well as both a left and right adjoint.

		\item We write $ \Mod_R(S) \colonequals S \tensor \Mod(R) $ and $ \Uup_{S} $ for the conservative left and right adjoint functor $ S \tensor \Uup \colon \fromto{\Mod_R(S)}{\Stab(S)} $.
	\end{enumerate}
\end{notation}

\begin{nul}\label{nul:ModR_forget}	
	Let $ \plowerstar \colon \fromto{S}{T} $ be a right adjoint between presentable \categories.
	As a consequence of \Cref{obs:adjointcompat}, the induced functors after tensoring with $ R $-module spectra commute with the forgetful functors in the sense that we have canonical identifications
	\begin{equation*}
		\Uup_{T} \of \Mod_R(\plowerstar)  = \Stab(\plowerstar) \of \Uup_{S} \andeq \Uup_{T} \of \Mod_R(\pupperstar) = \Stab(\pupperstar) \of \Uup_{S} \period
	\end{equation*}
\end{nul}

\begin{corollary}\label{cor:R-module_adjointability}
	Let $ R $ be an $ \Eup_1 $-ring. 
	Assume that:
	\begin{enumerate}[label=\stlabel{cor:R-module_adjointability}, ref=\arabic*]
		\item\label{cor:R-module_adjointability.1} \eqref{sq:generalsquare} is an oriented square in $ \PrR $.

		\item\label{cor:R-module_adjointability.2} The left adjoints $ \fupperstar \colon \fromto{D}{B} $ and $ \fbarupperstar \colon \fromto{C}{A} $ are left exact.
	\end{enumerate}
	Then the exchange morphism associated to the oriented square of stable \categories
	\begin{equation*}\tag*{$ \Mod_R(\textup{\ref{sq:generalsquare}}) $}\label{sq:R-Mod_square}
		\begin{tikzcd}[column sep=5em, row sep=2.5em]
			\Mod_R(A) \arrow[r, "\Mod_R(\fbarlowerstar)" above] \arrow[d, "\Mod_R(\gbarlowerstar)" left] & \Mod_R(D) \arrow[d, "\Mod_R(\glowerstar)" right] \arrow[dl, phantom, "\Longleftarrow" sloped] \\ 
			\Mod_R(B) \arrow[r, "\Mod_R(\flowerstar)" below] & \Mod_R(C) \period
		\end{tikzcd}
	\end{equation*}
	is an equivalence when restricted to those objects $ X \in \Mod_R(C) $ such that $ \Uup_C(X) \in \Stab(C_{\Ex}) $.
\end{corollary}

\begin{proof}
	Since the forgetful functor $ \Uup_{B} \colon \fromto{\Mod_R(B)}{\Stab(B)} $ is conservative, it suffices to show that for all $ X \in \Mod_R(C) $ such that $ \Uup_C(X) \in \Stab(C_{\Ex}) $, the morphism
	\begin{equation*}
		\Uup_{B}\Ex(X) \colon \fromto{\Uup_{B} \of \Mod_R(\fupperstar) \of \Mod_R(\glowerstar)(X)}{\Uup_{B} \of \Mod_R(\gbarlowerstar) \of \Mod_R(\fbarupperstar)(X)} 
	\end{equation*}
	is an equivalence.
	In light of \Cref{nul:ModR_forget}, we see that the morphism $ \Uup_{B}\Ex(X) $ is equivalent to the morphism
	\begin{equation*}
		\Ex(\Uup_{C}X) \colon \fromto{\Stab(\fupperstar) \of \Stab(\glowerstar)(\Uup_{C}X)}{\Stab(\fbarlowerstar) \of \Stab(\gbarupperstar)(\Uup_{C}X)} 
	\end{equation*}
	in $ \Stab(B) $.
	\Cref{prop:stable_adjointability} completes the proof.
\end{proof}

\begin{example}
	The Gabber--Illusie basechange theorem for oriented fiber product squares of coherent topoi \cite[Exposé XI, Théorème 2.4]{MR3309086} is an immediate consequence of \Cref{prop:stable_adjointability} combined with our nonabealian version \cite[Theorem 7.1.7]{exodromy}.
	See \cite[Proposition 7.4.11]{exodromy}.
\end{example}

\begin{example}[(Proper Basechange for étale cohomology)]
	As in the topological setting, \Cref{cor:R-module_adjointability} and Chough's Nonabelian Proper Basechange Theorem for torsion étale sheaves of spaces \cite[Theorem 1.2 \& Remark 1.6]{Chough:Proper} imply the classical result for torsion abelian sheaves \cite[Exposé XII, Théorème 5.1]{MR50:7132}.
	Moreover, this allows one to generalize the coefficients to any bounded-above $ \Eup_1 $-ring with finite homotopy groups.
\end{example}



\DeclareFieldFormat{labelnumberwidth}{#1}
\printbibliography[keyword=alph, heading=references]
\DeclareFieldFormat{labelnumberwidth}{{#1\adddot\midsentence}}
\printbibliography[heading=none, notkeyword=alph]

\end{document}

%% file: document_macros.tex
\renewcommand{\Stab}{\operatorname{Sp}}

\renewcommand{\trun}{\uptau}

\newcommand{\BC}{\mathrm{BC}}
\newcommand{\BCp}{\BC_{p}}

\newcommand{\Cbf}{\mathbfit{C}}
\newcommand{\Cp}{\operatorname{C}_{p}}

\newcommand{\RP}{\mathbf{RP}}
\newcommand{\Spcfin}{\Spc^{\fin}}
\newcommand{\Spcfinpt}{\Spcfin_{\ast}}

\newcommand{\cpt}{\mathrm{c}}
\newcommand{\cptproj}{\cpt\mathrm{pr}}
\newcommand{\fp}{\ensuremath{\textup{fp}}}
\newcommand{\Klim}{\ensuremath{\Kcal\textup{-lim}}}

\newcommand{\Catfp}{\Cat_{\infty}^{\kern0.05em\fp}}
\newcommand{\CatKlim}{\Cat_{\infty}^{\kern0.05em\Klim}}
\newcommand{\Catlex}{\Cat_{\infty}^{\kern0.05em\lex}}
\newcommand{\FunKlim}{\Fun^{\Klim}}
\newcommand{\Catinfty}{\Cat_{\infty}}

\newcommand{\Ecpt}{E^{\cpt}}
\newcommand{\Ecptproj}{E^{\cptproj}}
\newcommand{\Ecptop}{E^{\cpt,\op}}
\newcommand{\Ecptprojop}{E^{\cptproj,\op}}

\newcommand{\st}{\mathrm{st}}
\newcommand{\PrLst}{\PrL_{\st}}

\renewcommand{\D}{\operatorname{D}}

\newcommand{\Lfupperstar}{\mathrm{L}\fupperstar}

\newcommand{\Lpupperstar}{\mathrm{L}\pupperstar}
\newcommand{\Lfbarupperstar}{\mathrm{L}\fbar\upperstar}

\newcommand{\Rflowerstar}{\mathrm{R}\flowerstar}
\newcommand{\Rglowerstar}{\mathrm{R}\glowerstar}

\newcommand{\Rfbarlowerstar}{\mathrm{R}\fbar\lowerstar}
\newcommand{\Rgbarlowerstar}{\mathrm{R}\gbar\lowerstar}

\newcommand{\guppersharp}{g^{\sharp}}
\newcommand{\gbaruppersharp}{\gbar^{\sharp}}